\documentclass[reqno]{amsart}
\usepackage{amssymb}
\usepackage{amscd}
\usepackage[dvips]{graphics}
\usepackage{mathrsfs}
\usepackage{bm}
\usepackage{cite}
\usepackage{color}
\usepackage{hyperref}

\def\beq{\begin{equation}}
\def\eeq{\end{equation}}
\def\ba{\begin{array}}
\def\ea{\end{array}}

\numberwithin{equation}{section}
\newtheorem{theorem}{Theorem}[section]

\newtheorem{proposition}[theorem]{\textbf{Proposition}}
\newtheorem{lemma}[theorem]{Lemma}
\newtheorem{thm}{Theorem} 
\newtheorem{rem}{\textbf{Remark}}[thm]
\renewenvironment{proof}{\noindent{\textbf{Proof.}}}{\hfill$\Box$}
\theoremstyle{remark}
\newtheorem{remark}[theorem]{\textbf{Remark}}
\theoremstyle{plain}

\allowdisplaybreaks[4]

\begin{document}
\title[Reversed HLS inequality on $\mathbb{H}^n$ and $\mathbb{S}^{2n+1}$]{Reversed Hardy-Littlewood-Sobolev inequality on Heisenberg group $\mathbb{H}^n$ and CR sphere $\mathbb{S}^{2n+1}$}


\author{Yazhou Han}
\address{Yazhou Han, Department of Mathematics, College of Science, China Jiliang University, Hangzhou, 310018, China} \email{yazhou.han@gmail.com}


\author{Shutao Zhang}
\address{Shutao Zhang, Department of Mathematics, College of Science, China Jiliang University, Hangzhou, 310018, China} \email{zhangst@cjlu.edu.cn}




\begin{abstract}
    This paper is mainly devoted to the study of the reversed Hardy-Littlewood-Sobolev (HLS) inequality on Heisenberg group $\mathbb{H}^n$ and CR sphere $\mathbb{S}^{2n+1}$. First, we establish the roughly reversed HLS inequality and give a explicitly lower bound for the sharp constant.
    Then, the existence of the extremal functions with sharp constant is proved by \textit{subcritical approach} and some compactness techniques. 
    Our method is \textit{rearrangement free} and can be applied to study the classical HLS inequality and other similar inequalities.
\end{abstract}
\maketitle

\section{Introduction}

Heisenberg group is one of the simplest noncommutative geometries and is the model space of CR manifolds, which arise from the study of real hypersurfaces of complex manifolds. 
It is well-known that the non-commutativity and the complex structure induced from complex manifolds inspire many interesting geometric properties and bring some new difficulties. In the past few decades,
sharp inequalities such as Sobolev inequality \cite{Folland-Stein1974, JL1988, Folland1975, Han-Zhang2020}, Hardy-Littlewood-Sobolev(HLS) inequality \cite{Folland-Stein1974, Frank-Lieb2012}, 
Moser-Trudinger inequality \cite{BFM2013, Cohn-Lu2004, Cohn-Lu2001}, Hardy inequality \cite{Garofalo-Lanconelli1990, Niu-Zhang-Wang01}, Hardy-Sobolev inequality \cite{Han-Niu2005}, etc., play important roles in the study of problems defined on Heisenberg group and CR manifolds. In this paper, we mainly concern with \textit{reversed Hardy-Littlewood-Sobolev inequality} on $\mathbb{H}^n$ and CR sphere $\mathbb{S}^{2n+1}$.



\subsection{HLS and reversed HLS inequalities on $\mathbb{R}^n$}
The classical HLS inequality \cite{HL1928, HL1930, So1963} on $\mathbb{R}^n$ states that
\begin{equation}\label{HLS Rn}
  \left|\int_{\mathbb{R}^n}\int_{\mathbb{R}^n} \frac{f(x) g(y)}{|x-y|^{n-\alpha}}dxdy\right|\leq N_{p,\alpha,n}\|f\|_p\|g\|_t
\end{equation}
holds for all $f\in L^p(\mathbb{R}^n),\ g\in L^t(\mathbb{R}^n)$, where $0<\alpha<n$ and $1<p,t<+\infty$ satisfying
\begin{equation}\label{HLS relation Rn}
  \frac 1p+\frac 1t+\frac{n-\alpha}{n}=2.
\end{equation}
Using rearrangement inequalities, Lieb \cite{Lieb1983} proved the existence of the extremal functions to the inequality \eqref{HLS Rn} with the sharp constant. 
For the conformal case $p=t=\frac{2n}{n+\alpha}$, 
he classified the extremal functions and computed the best constant (different discussions can be found in \cite{Lieb-Loss2001, Carlen-Loss}). In fact, he proved that the extremal functions with $p=t$ are given by 
\begin{equation}\label{bubbling solution of HLS}
    f_\epsilon(x)=c_1g_\epsilon(x)=c\left(\frac \epsilon {\epsilon^2+|x-x_0|^2}\right)^{(n+\alpha)/2},
\end{equation}
where $c_1,c$ and $\epsilon$ are constants, $x_0$ is some point in $\mathbb{R}^n$. Recently, the solutions of the Euler-Lagrange equation in the conformal case were classified by the method of moving planes\cite{CLO2006} and the method of moving spheres\cite{Li2004}, respectively.

For $0<p,t<1$ and $\alpha>n$ satisfying \eqref{HLS relation Rn}, Dou and Zhu \cite{DZ2015} (also see \cite{B2015, Ngo-Nguyen2107}) established a class of reversed HLS inequality  
\begin{equation}\label{HLS re Rn}
  \int_{\mathbb{R}^n}\int_{\mathbb{R}^n} \frac{f(x) g(y)}{|x-y|^{n-\alpha}}dxdy\geq N_{p,\alpha,n}\|f\|_p\|g\|_t,
\end{equation}
where $f\in L^p(\mathbb{R}^n),\ g\in L^t(\mathbb{R}^n)$ are nonnegative functions. Employing the rearrangement inequalities and the method of moving spheres, they also classified the extremal functions and computed the best constant in the conformal case. In fact, they found that the extremal functions of \eqref{HLS re Rn} in the conformal case are given as \eqref{bubbling solution of HLS}, too.

%
%

As stated above, it can be found that \textit{rearrangement inequalities, the method of moving planes and the method of moving spheres} are basic and important tools in the study of HLS inequalities. More applications of these techniques  can be found in the study of HLS inequalities and reversed HLS inequalities on the upper half space (see \cite{DZ2015a, Dou-Guo-Zhu2017,  Ngo-Nguyen2017b,  HWY2007} and the references therein).

Note that $f_\epsilon$ and $g_\epsilon$ will blow up as $\epsilon\rightarrow 0^+$, and vanish as $\epsilon\rightarrow +\infty$. The phenomenon makes it difficult to study the extremal problems. 
To overcome the difficulty, we often renormalize  the extremal sequence. For example, Lieb\cite{Lieb1983} renormalized the extremal sequence $\{f_j(x)\}$ so that it satisfies $f_j(x)>\beta>0$ if $|x|=1$. The technique can also be found in \cite{DZ2015}.

Recently, Dou, Guo and Zhu\cite{Dou-Guo-Zhu2017} adopted the \textit{subcritical approach} to study sharp HLS type inequalities on the upper half space. By Young inequality, they first established two classes of HLS type inequalities with subcritical power on a ball. Then, using the conformal transformation between ball and upper half space and the method of moving planes, they proved that the extremal functions of HLS type inequalities with subcritical power are constant functions. Passing to the limit from subcritical power to critical power, they obtained two classes of sharp HLS type inequalities on the upper half space.
In the process of taking the limit, since these extremal functions of HLS type inequalities with subcritical power are constant functions, we can choose every extremal function to be $f\equiv 1$ and avoid efficiently the blow-up phenomenon.



\subsection{HLS inequlity on the Heisenberg group} We first recall Heisenberg group and some notations.

Heisenberg group $\mathbb{H}^n$ consists of the set
$$\mathbb{C}^n \times \mathbb{R}=\{(z,t):z=(z_1, \cdots, z_n)\in \mathbb{C}^n,t\in  \mathbb{R}\}$$
with the multiplication law
$$
(z,t)(z',t')=(z+z', t+t'+2Im(z \cdot \overline{z'})),
$$
where $z \cdot \overline{z'}=\sum_{j=1}^n z_j \overline{z_j'}$, $z_j=x_j+\sqrt{-1}y_j$ and $\overline{z_j}=x_j-\sqrt{-1}y_j$.

For any points $u=(z,t),\ v=(z',t')\in\mathbb{H}^n$, denote the norm function by $|u|=(|z|^4+t^2)^{1/4}$ and the distance between $u$ and $v$ by $|v^{-1}u|$. Moreover, there exists a constant $\gamma\geq 1$ such that $|uv|\leq \gamma(|u|+|v|)$ holds for all $u,v\in\mathbb{H}^n$. Write $B(u,R)=\{v\in\mathbb{H}^n:\ |u^{-1}v|<R\}$  as the ball centered at $u$ with radius $R$. For any $\lambda>0$, the dilation $\delta_\lambda(u)$ is defined as $\delta_\lambda(u)=(\lambda z,\lambda^2 t)$, and $Q=2n+2$  is the homogeneous dimension with respect to the dilations. For more details about Heisenberg group, please see \cite{Dragomir-Tomassini2006, Folland-Stein1974} and the references therein.

To study the sigular integral operator on CR manifolds, Folland and Stein \cite{Folland-Stein1974} established the following HLS inequality
\begin{equation}\label{HLS Hn roughly}
    \left|\int_{\mathbb{H}^n}\int_{\mathbb{H}^n} \overline{f(u)}g(v)|v^{-1}u|^{\alpha-Q}dv du\right| \leq D(n,\alpha,p)\|f\|_{L^{q}(\mathbb{H}^n)} \|g\|_{L^p(\mathbb{H}^n)},
\end{equation}
where $f\in L^{q},\ g\in L^p$, $0<\alpha<Q$, $\frac 1{q}+\frac 1p+\frac{Q-\alpha}Q=2$ and $du=dzdt=dxdydt$ is the Haar measure on $\mathbb{H}^n$. In fact, the inequality \eqref{HLS Hn roughly} can be deduced from Proposition 8.7 of \cite{Folland-Stein1974}. 

Since rearrangement inequalities do not work efficiently on Heisenberg group, it took a quite long time to study the problems about the sharp constant and extremal functions  of \eqref{HLS Hn roughly}. In 2012, Frank and Lieb \cite{Frank-Lieb2012} studied the conformal case $p=q=\frac{2Q}{Q+\alpha}$ of \eqref{HLS Hn roughly}. They introduced a class of \textit{rearrangement free method} and classified the extremal functions. 
Then, sharp constants were computed for the HLS inequality, Sobolev inequality and their limiting cases on Heisenberg group and CR sphere $\mathbb{S}^{2n+1}$. Their results about HLS inequality on $\mathbb{H}^n$  can be stated as follows.
\begin{thm}[Sharp HLS inequality on $\mathbb{H}^n$]\label{thm HLS Hn sharp}
Let $0<\alpha<Q$ and $p_\alpha=\frac{2Q}{Q+\alpha}$. Then for any $f,g\in L^{p_\alpha}(\mathbb{H}^n)$,
\begin{equation}\label{HLS Hn sharp}
    \left|\int_{\mathbb{H}^n} \int_{\mathbb{H}^n} \overline{f(u)}|v^{-1}u|^{-(Q-\alpha)} g(v) dv du\right|\le D_{n,\alpha} ||f||_{L^{p_\alpha}(\mathbb{H}^n)} ||g||_{L^{p_\alpha}(\mathbb{H}^n)},
\end{equation}
where
\begin{equation}\label{extreCon}
D_{n,\alpha}:=\left(\frac{\pi^{n+1}}{2^{n-1}n!}\right)^{(Q-\alpha)/Q} \frac{n!\Gamma(\alpha/2)}{\Gamma^2((Q+\alpha)/4)}.
\end{equation}
And the equality holds if and only if
\begin{equation}\label{HLS-ex}
f(u)=c_1g(u)=c_2H(\delta_r(u_0^{-1}u),
\end{equation}
for some $c_1, \ c_2\in\mathbb{C}$,\ $r>0$ and $u_0\in\mathbb{H}^n$ (unless $f\equiv 0$ or $g\equiv 0$). Here $H$ is defined as
\begin{equation}\label{HLS-ex1}
    H(u)=H(z,t)=((1+|z|^2)^2+t^2)^{-(Q+\alpha)/4}.
\end{equation}
\end{thm}

\begin{rem}
Using the Green's function of the sub-Laplacian\cite{Folland1973} and making a duality argument, we see that HLS inequality \eqref{HLS Hn sharp} with $\alpha=2$ is equivalent to the sharp Sobolev inequality established by Jeison and Lee\cite{JL1988}. Based on the idea introduced by Obata\cite{Obata1971}, they classified the extremal functions and computed the sharp constant of the sharp Sobolev inequality (see \cite{JL1988}).
\end{rem}

In view of the efficiency of the method of moving planes and the method of moving spheres in the study of Euler-Lagrange equation of \eqref{HLS Rn}, a natural question is whether one can adapt them on the Heisenberg group. There have been a number of attempts by several mathematicians in the directions (see \cite{BP1999, Han-Wang-Zhu2017} and the references therein). But, it seems that these methods are not suitable very well with Heisenberg group.

For the case $p\neq q$, Han \cite{H2013} used the concentration-compactness principles to study the existence of extremal functions of \eqref{HLS Hn roughly}. Recently, Han, Lu, Zhu \cite{HLZ2012} and Chen, Lu, Tao \cite{CLT2019} established two classes of  weighted HLS inequalities on Heisenberg group and proved the existence of extremal functions  by the concentration-compactness principles.



\subsection{Reversed HLS inequalities on the Heisenberg group}
If $\alpha>Q$, we will establish the following reversed HLS inequality.
\begin{proposition}
Let $\alpha>Q\geq 4$ and $p_\alpha=\frac{2Q}{Q+\alpha}$. Then for any nonnegative functions $f,g\in L^{p_\alpha}(\mathbb{H}^n)$, there exists a sharp constant $N_{Q,\alpha,\mathbb{H}}$ such that
\begin{equation}\label{re HLS conformal roughly}
    \int_{\mathbb{H}^n}\int_{\mathbb{H}^n} \frac{F(u)G(v)} {|v^{-1}u|^{Q-\alpha}} du dv\geq N_{Q,\alpha,\mathbb{H}} \|F\|_{L^{p_\alpha}(\mathbb{H}^n)} \|G\|_{L^{p_\alpha}(\mathbb{H}^n)}.
\end{equation}
The sharp constant satisfies
$$N_{Q,\alpha,\mathbb{H}}\geq \frac{(8|B_1|)^{(Q-\alpha)/Q}}{2p_\alpha^2},$$
where $B_1:=B(0,1)$ 
and the volume of $B_1$ is given (see \cite{Cohn-Lu2001,H2013}) as $$|B_1|=\int_{|u|<1}du=\frac{2\pi^{\frac{Q-2}2}\Gamma(\frac 12)\Gamma(\frac{Q+2}{4})} {(Q-2)\Gamma(\frac{Q-2}{2})\Gamma(\frac{Q+4}{4})}.$$
\end{proposition}

Since $p_\alpha\in(0,1)$, 
the extremal problem of \eqref{re HLS conformal roughly} is analytically different from the case $\alpha\in(0,Q)$. This brings some difficulties to study the case $\alpha>Q$ by the method of \cite{Frank-Lieb2012} and \cite{H2013}.

We will discuss the extremal problem by \textit{subcritical approach}. However, because of the non-commutativity and the complex structure of Heisenberg group and CR sphere, which make the method of moving planes and the method of  moving spheres ineffective, 
it is not easy to prove that the extremal functions of HLS inequalities with subcritical power on the CR sphere should be constant functions. We will encounter the blow-up phenomenon and circumvent it by \textit{renormalization method} (see Section \ref{sec critical}). Furthermore, our method is rearrangement free and different from the method in \cite{Frank-Lieb2012}. Recently, we have successfully experimented with the method and provided a new proof for the existence of extremal functions of \eqref{HLS Rn} and \eqref{HLS re Rn} (see \cite{zhang-Han2}).

\medskip


The unit CR sphere is the sphere $\mathbb{S}^{2n+1}=\{\xi=(\xi_1, \cdots, \xi_{n+1} \in \mathbb{C}^{n+1}:\ \|\xi\|=1\}$ endowed with standard CR structure. \textit{Cayley transformation} $\mathcal{C}: \mathbb{H}^n\rightarrow \mathbb{S}^{2n+1}\setminus\mathfrak{S}$ and its reverse are defined respectively as
\begin{gather*}
    C(z,t)=\Bigl(\frac{2z}{1+|z|^2+it},\frac{1-|z|^2-it}{1+|z|^2+it}\Bigr),\\
    C^{-1}(\xi) =\Bigl( \frac{\xi_1}{1+\xi_{n+1}}, \cdots, \frac{\xi_n}{1+\xi_{n+1}}, \rm{Im}\frac{1-\xi_{n+1}}{1+\xi_{n+1}}\Bigr),
\end{gather*}
where $\mathfrak{S}=(0,\cdots,0,-1)$ is the south pole. The Jacobian of the Cayley transformation is
    $$J_\mathcal{C}{(z,t)}=\frac{2^{2n+1}}{((1+|z|^2)^2+t^2)^{n+1}}$$
which implies that
\begin{equation}\label{transform integral Hn S}
    \int_{\mathbb{S}^{2n+1}} \phi(\xi) d\xi=\int_{\mathbb{H}^n} \phi(\mathcal{C}(u)) J_\mathcal{C}(u) du
\end{equation}
for all integrable function $\phi$ on $\mathbb{S}^{2n+1}$, where 
$d\xi$ is the Euclidean volume element of $\mathbb{S}^{2n+1}$. Under the Cayley transformation, we have the following relations between two distance functions
\begin{equation}\label{formula relations distance}
    |1-\xi\cdot\bar\eta|=2|u^{-1}v|^2\left( {(1+|z|^2)^2+t^2}\right)^{-1/2} \left({(1+|z'|^2)^2+t'^2}\right)^{-1/2},
\end{equation}
where $\zeta=\mathcal{C}(u),\ \eta=\mathcal{C}(v),\ u=(z,t)$ and $v=(z',t')$.

For any $f\in L^p(\mathbb{S}^{2n+1})$, there is a corresponding function
    $$F(u)=|J_\mathcal{C}(u)|^{1/p} f(\mathcal{C}(u))\in L^p(\mathbb{H}^n)$$
such that $\|f\|_{L^p(\mathbb{S}^{2n+1})}= \|F\|_{L^p(\mathbb{H}^{n})}$.

Applying the Cayley transformation to \eqref{re HLS conformal roughly}, we have that 
\begin{equation}\label{re HLS roughly S}
    \int_{\mathbb{S}^{2n+1}}\int_{\mathbb{S}^{2n+1}} \frac {f(\xi) g(\eta)} {|1-\xi\cdot\bar\eta|^{\frac{Q-\alpha}2}} d\xi d\eta\geq N_{Q,\alpha}  \|f\|_{L^{p_\alpha}(\mathbb{S}^{2n+1})} \|g\|_{L^{p_\alpha}(\mathbb{S}^{2n+1})}
\end{equation}
holds for all nonnegative functions $f,g\in L^{p_\alpha}(\mathbb{S}^{2n+1})$, where $N_{Q,\alpha}$ is the sharp constant and satifies
    $$N_{Q,\alpha}\geq  \frac{(8|B_1|)^{\frac{Q-\alpha}Q}} {2^{1+n\frac{\alpha-Q}Q} p_\alpha^2}.$$

Define the extremal problem of \eqref{re HLS roughly S} as
\begin{align}\label{extremal reversed}
    N_{Q,\alpha}
    &=\inf_{\|f\|_{L^{p_\alpha}(\mathbb{S}^{2n+1})}= \|g\|_{L^{p_\alpha}(\mathbb{S}^{2n+1})}=1} \int_{\mathbb{S}^{2n+1}}\int_{\mathbb{S}^{2n+1}} \frac {f(\xi) g(\eta)} {|1-\xi\cdot\bar\eta|^{(Q-\alpha)/2}} d\xi d\eta\nonumber\\
    &=\inf_{f,g\in L^{p_\alpha}(\mathbb{S}^{2n+1})\setminus\{0\}} \frac{\int_{\mathbb{S}^{2n+1}} \int_{\mathbb{S}^{2n+1}} {f(\xi) g(\eta)} {|1-\xi\cdot\bar\eta|^{(\alpha-Q)/2}} d\xi d\eta} {\|f\|_{L^{p_\alpha}(\mathbb{S}^{2n+1})}  \|g\|_{L^{p_\alpha}(\mathbb{S}^{2n+1})}}.
\end{align}
Then, it is easy to get the following estimate.
\begin{proposition}[Upper and lower bound for the sharp constant]
\begin{equation}\label{estimate of sharp constant}
  0 < \frac{(8|B_1|)^{\frac{Q-\alpha}Q}} {2^{1+n\frac{\alpha-Q}Q} p_\alpha^2} \leq N_{Q,\alpha} \leq \Bigl( \frac{2\pi^{n+1}}{n!} \Bigr)^{(Q-\alpha)/Q} \frac{n! \Gamma(\alpha/2)}{\Gamma^2((Q+\alpha)/2)},
\end{equation}
where
  \begin{equation}\label{sharp constant hoped}
     \Bigl( \frac{2\pi^{n+1}}{n!} \Bigr)^{(Q-\alpha)/Q} \frac{n! \Gamma(\alpha/2)}{\Gamma^2((Q+\alpha)/2)}=|\mathbb{S}^{2n+1}|^{1-\frac 2{p_\alpha}}\int_{\mathbb{S}^{2n+1}} |1-\xi\cdot\bar\eta|^{\frac{\alpha-Q}{2}}d\eta.
  \end{equation}
\end{proposition}



Combining the subcritical approach and renormalization method, we prove the following attainability of the sharp constant $N_{Q,\alpha}$.

\begin{theorem}[Attainability]\label{pro HLS exist}
$N_{Q,\alpha}$ can be attained by a pair of positive functions $f,g\in C^1(\mathbb{S}^{2n+1})$. Applying the Cayley transformation, we also have that $N_{Q,\alpha,\mathbb{H}}$ is attained by a pair of positive functions $F,G\in L^{p_\alpha}(\mathbb{H}^n)\cap C^1(\mathbb{H}^{n})$.
\end{theorem}

In the following, we outline the ideas of the proof of Theorem \ref{pro HLS exist}. 
First, consider the extremal problems with subcritical power $p\in (0,p_\alpha)$ and get the existence of extremal function pairs $\{f_p, g_p\}$, 
see Section \ref{Sec sub}. Then, prove that the sequence $\{f_p, g_p\}$ form a minimizing sequence of \eqref{extremal reversed} as $p\rightarrow p_\alpha$. Lastly, we circumvent the blow-up phenomenon by renormalization method and show the attainability of the sharp constant $N_{Q,\alpha}$.

Moreover, since nonlinear terms with negative power appear in the Euler-Lagrange equations (see Section \ref{Sec sub} and Section \ref{sec critical}), we need not only a upper bound to control the blow up of the sequence, but also a lower bound to avoid the blow up of terms with negative power. So, different techniques are needed for the extremal problem \eqref{extremal reversed}. More details can be seen in Section \ref{Sec sub} and Section \ref{sec critical}.


The paper is organized as follows. Section \ref{sec rough} is devoted to establishing the roughly reversed HLS inequalities \eqref{re HLS conformal roughly}. In Section \ref{Sec sub}, we  study the extremal problems related to subcritical case and get the existence of the corresponding extremal functions. These functions will provide a minimizing sequence of \eqref{extremal reversed}. Then, we   prove the attainability of $N_{Q,\alpha}$ in Section \ref{sec critical}.

We always use $C, C_1,C_2,\cdots$, etc. to denote positive universal constants though their actual values may differ from line to line or within the same line itself.

\section{Roughly reversed HLS inequalities on $\mathbb{H}^n$}\label{sec rough}

This section is mainly devoted to establishing the roughly reversed HLS inequality \eqref{re HLS conformal roughly}. In fact, we present  a more general reversed HLS inequalities as follows.
\begin{proposition}\label{pro sub HLS}
Assume $\lambda>0$, $0<p,t<1$  with $\frac 1p+\frac 1t-\frac\lambda Q=2$. Then, for any nonnegative functions $F\in L^p(\mathbb{H}^n)$ and $G\in L^t(\mathbb{H}^n)$, there exists some positive constant $C(Q,\lambda,p,\mathbb{H})$ such that
\begin{equation}\label{re HLS roughly}
    \int_{\mathbb{H}^n}\int_{\mathbb{H}^n} F(u) |v^{-1}u|^\lambda G(v) du dv\geq C(Q,\lambda,p,\mathbb{H}) \|F\|_{L^p(\mathbb{H}^n)} \|G\|_{L^t(\mathbb{H}^n)}.
\end{equation}
Moreover, the constant satisfies
\begin{equation}\label{lower constant}
    C(Q,\lambda,p,\mathbb{H})\geq\frac{(4|B_1|)^{-\lambda/Q}}{2pt}\Bigl( \frac\lambda Q \max\{ \frac p{1-p}, \frac t{1-t} \}\Bigr)^{-\lambda/Q}.
\end{equation}
\end{proposition}
\begin{proof}
Our proof is similar to the argument given by Ng\^{o} and Nguyen \cite[Section 2]{Ngo-Nguyen2107}, where authors adopted the idea of Lieb and Loss \cite{Lieb-Loss2001}. For completeness, we will give the detailed proof. Since the homogeneity of \eqref{re HLS roughly}, without loss of generality, we assume that $\|F\|_{L^p(\mathbb{H}^n)}=\|G\|_{L^t(\mathbb{H}^n)}=1$. So, it is sufficient to show that 
the right side of \eqref{lower constant} is a lower bound of the left side of \eqref{re HLS roughly}.

By the  layer cake representation \cite[Theorem 1.13]{Lieb-Loss2001},
\begin{align}\label{rough 1}
    I:=& \int_{\mathbb{H}^n}\int_{\mathbb{H}^n} F(u) |v^{-1}u|^\lambda G(v) du dv\nonumber\\
    =& \lambda \int_0^{\infty} \int_0^{\infty} \int_0^{\infty} c^{\lambda-1} J(a,b,c) da\ db\ dc,
\end{align}
where
\begin{equation*}
    J(a,b,c):= \int_{\mathbb{H}^n}\int_{\mathbb{H}^n} \chi_{\{F>a\}}(u) \chi_{\mathbb{H}^n \setminus B_c}(u^{-1}v) \chi_{\{G>b\}}(v) du\ dv
\end{equation*}
and $\chi_\Omega(u)$ is the characteristic function of set $\Omega$, $B_c:=B(0,c)$. Noting the basic fact $|u^{-1}v|=|v^{-1}u|$,  we have $\chi_{\mathbb{H}^n \setminus B_c}(u^{-1}v)= \chi_{\mathbb{H}^n \setminus B_c}(v^{-1}u)$. Write $\phi(a)=\int_{\mathbb{H}^n} \chi_{\{F>a\}}(u) du=|\{F>a\}|$ and $\psi(b)=\int_{\mathbb{H}^n} \chi_{\{G>b\}}(v) dv=|\{G>b\}|$.

If $\phi(a)\geq \psi(b)$ and $2|B_c|=2C^Q|B_1| \leq \phi(a)$, then
\begin{align*}
    J(a,b,c)=& \int_{\mathbb{H}^n} \chi_{\{G>b\}}(v) |\{F>a\} \cap (\mathbb{H}^n \setminus B_c(v))|dv\\
    \geq & \int_{\mathbb{H}^n} \chi_{\{G>b\}}(v) \big(|\{F>a\}|- |B_c(v)|\bigr)dv\\
    \geq& \int_{\mathbb{H}^n} \chi_{\{G>b\}}(v) \frac{\phi(a)}2 dv=\frac{\phi(a)\psi(b)}2.
\end{align*}
Similarly, if $\phi(a)\leq \psi(b)$ and $2|B_c|=2C^Q|B_1| \leq \psi(b)$, the above formula also holds. Therefore, if $2|B_c|=2C^Q|B_1| \leq \max\{\phi(a), \psi(b)\}$,  it follows
\begin{equation}\label{rough 2}
    J(a,b,c)\geq \frac{\phi(a)\psi(b)}2.
\end{equation}
Substituting \eqref{rough 2} into \eqref{rough 1}, we have
\begin{align}\label{rough 3}
    I\geq & \lambda \int_0^{\infty} \int_0^{\infty} \Bigl(\int_0^{(\frac{\max\{\phi(a), \psi(b)\}}{2|B_1|})^{1/Q}} c^{\lambda-1} \frac{\phi(a)\psi(b)}2 dc\Bigr) da\ db\nonumber\\
    =& \int_0^{\infty} \int_0^{\infty} \frac{\phi(a)\psi(b)}2  \Bigl(\frac{\max\{\phi(a), \psi(b)\}}{2|B_1|}\Bigr)^{\frac\lambda Q}da\ db\nonumber\\
    \geq & \frac{(2|B_1|)^{-\lambda/Q}}2 \int_0^{\infty} \int_0^{a^{p/t}} \phi(a)\psi(b)^{1+\frac\lambda Q} db\ da\nonumber\\
    &+\frac{(2|B_1|)^{-\lambda/Q}}2 \int_0^{\infty} \int_{a^{p/t}}^\infty \phi(a)^{1+\frac\lambda Q}\psi(b) db\ da\nonumber\\
=:&\frac{(2|B_1|)^{-\lambda/Q}}2(I_1+I_2).
\end{align}
By reversed H\"{o}lder inequality, it yields
\begin{align}\label{rough 4}
    I_1=& \int_0^{\infty} \int_0^{a^{p/t}} \phi(a)\psi(b)^{1+\frac\lambda Q} db\ da\nonumber\\
    \geq & \int_0^{\infty} \phi(a) \Bigl(\int_0^{a^{p/t}} \psi(b) b^{t-1} db \Bigr)^{\frac{Q+\lambda}Q} \Bigl(\int_0^{a^{p/t}} b^{(t-1) \frac{Q+\lambda}\lambda} db\Bigr)^{-\frac\lambda Q} da\nonumber\\
    =& \int_0^{\infty} \phi(a) \Bigl(\int_0^{a^{p/t}} \psi(b) b^{t-1} db \Bigr)^{\frac{Q+\lambda}Q} \Bigl( \frac{\lambda p}{Qt(1-p)} \Bigr)^{-\frac\lambda Q} a^{p-1} da\nonumber\\
    =& \frac{1}{pt}(\frac\lambda Q\frac{p}{1-p})^{-\frac\lambda Q} \int_0^\infty p a^{p-1} \phi(a) \Bigl(\int_0^{a^{p/t}} t b^{t-1} \psi(b) db \Bigr)^{\frac{Q+\lambda}Q} da
\end{align}
and
\begin{align}\label{rough 5}
    I_2=& \int_0^{\infty} \int_{a^{p/t}}^\infty \phi(a)^{1+\frac\lambda Q}\psi(b) db\ da= \int_0^{\infty} \int_0^{b^{t/p}} \phi(a)^{1+\frac\lambda Q}\psi(b) da\ db\nonumber\\
    \geq & \frac{1}{pt}(\frac\lambda Q\frac{t}{1-t})^{-\frac\lambda Q} \int_0^\infty t b^{t-1} \psi(b) \Bigl(\int_0^{b^{t/p}} p a^{p-1} \phi(a) da \Bigr)^{\frac{Q+\lambda}Q} db.
\end{align}
Noting that
\begin{gather*}
    1=\|F\|_{L^p(\mathbb{H}^n)}^p=p\int_0^\infty a^{p-1} \phi(a) da,\\
    1=\|G\|_{L^t(\mathbb{H}^n)}^t=t\int_0^\infty b^{t-1} \psi(b)db
\end{gather*}
and $\frac{Q+\lambda}Q\geq 1$, it follows from Jensen inequality that
\begin{align}
    I_1\geq & \frac{1}{pt}(\frac\lambda Q\frac{p}{1-p})^{-\frac\lambda Q} \Bigl(\int_0^\infty p a^{p-1} \phi(a) \int_0^{a^{p/t}} t b^{t-1} \psi(b) db\ da\Bigr)^{\frac{Q+\lambda}Q},\label{rough 6}\\
    I_2\geq & \frac{1}{pt}(\frac\lambda Q\frac{t}{1-t})^{-\frac\lambda Q} \Bigl(\int_0^\infty t b^{t-1} \psi(b)\int_0^{b^{t/p}} p a^{p-1} \phi(a) da\ db\Bigr)^{\frac{Q+\lambda}Q}\nonumber\\
    =& \frac{1}{pt}(\frac\lambda Q\frac{t}{1-t})^{-\frac\lambda Q} \Bigl(\int_0^\infty p a^{p-1} \phi(a) \int_{a^{p/t}}^\infty t b^{t-1} \psi(b) db\ da\Bigr)^{\frac{Q+\lambda}Q}.\label{rough 7}
\end{align}
Write $$C_1(Q,\lambda,p):=\frac{(2|B_1|)^{-\lambda/Q}}{2pt} \Bigl(\frac\lambda Q \max\{\frac{p}{1-p},\frac{t}{1-t}\}\Bigr)^{-\frac\lambda Q}.$$
Substituting \eqref{rough 6} and \eqref{rough 7} into \eqref{rough 3} and using the convexity of function $x^{\frac{Q+\lambda}Q}$,  we arrive at
\begin{align*}
    I\geq & C_1(Q,\lambda,p)\Bigl(\int_0^\infty p a^{p-1} \phi(a) \int_0^{a^{p/t}} t b^{t-1} \psi(b) db\ da\Bigr)^{\frac{Q+\lambda}Q}\\
    &+C_1(Q,\lambda,p)\Bigl(\int_0^\infty p a^{p-1} \phi(a) \int_{a^{p/t}}^\infty t b^{t-1} \psi(b) db\ da\Bigr)^{\frac{Q+\lambda}Q}\\
    \geq& C_1(Q,\lambda,p) 2^{-\frac\lambda Q}.
\end{align*}
The inequality \eqref{re HLS roughly} is established and the proof is completed.
\end{proof}

\begin{remark}  The inequality \eqref{re HLS roughly} includes the inequalty \eqref{re HLS conformal roughly}. In fact,
suppose that $\lambda=\alpha-Q$ with $\alpha>Q\geq 4$ and $p=t=p_\alpha$,  \eqref{re HLS roughly} is reduced to \eqref{re HLS conformal roughly}.
\end{remark}

\section{Subcritical HLS inequalities on $\mathbb{S}^{2n+1}$}\label{Sec sub}

\begin{lemma}\label{sub reversed HLS pro}
Let $p\in (0,p_\alpha)$. There exists some positive constant $\tilde{C}=C(Q,\alpha,p)$ such that
\begin{equation}\label{sub reversed HLS Sn}
    \int_{\mathbb{S}^{2n+1}}\int_{\mathbb{S}^{2n+1}} \frac {f(\xi) g(\eta)} {|1-\xi\cdot\bar\eta|^{(Q-\alpha)/2}} d\xi d\eta\geq  \tilde{C} \|f\|_{L^p(\mathbb{S}^{2n+1})}\|g\|_{L^p(\mathbb{S}^{2n+1})}
\end{equation}
holds for any nonnegative $f,g\in L^p(\mathbb{S}^{2n+1})$.
\end{lemma}
\begin{proof}
It is easy to verify that \eqref{sub reversed HLS Sn} holds for any nonnegative $f,g\in L^p(\mathbb{S}^{2n+1})\cap L^{p_\alpha}(\mathbb{S}^{2n+1})$ by 
\eqref{re HLS roughly S} and H\"{o}lder inequality. Then we complete the proof by a density argument.
\end{proof}

Define the extremal problem of \eqref{sub reversed HLS Sn} as
\begin{align}\label{extremal sub reversed}
    N_{Q,\alpha,p}
    &=\inf_{\|f\|_{L^p(\mathbb{S}^{2n+1})}= \|g\|_{L^p(\mathbb{S}^{2n+1})}=1} \int_{\mathbb{S}^{2n+1}}\int_{\mathbb{S}^{2n+1}} \frac {f(\xi) g(\eta)} {|1-\xi\cdot\bar\eta|^{(Q-\alpha)/2}} d\xi d\eta.
\end{align}
From  \eqref{sub reversed HLS Sn}, it is easy to see that
\begin{align}\label{sub re-0}
  0<\tilde{C}\leq N_{Q,\alpha,p} & \leq |\mathbb{S}^{2n+1}|^{1-\frac 2p}\int_{\mathbb{S}^{2n+1}} |1-\xi\cdot\bar\eta|^{\frac{\alpha-Q}{2}}d\eta \nonumber\\
  & = \Bigl( \frac{2\pi^{n+1}}{n!} \Bigr)^{2-\frac 2p} \frac{n! \Gamma(\alpha/2)}{\Gamma^2((Q+\alpha)/2)}.
\end{align}
Furthermore, inspired by the argument of Lemma 3.2 of \cite{Dou-Guo-Zhu} and Proposition 2.5 of \cite{Dou-Guo-Zhu2017}, we will prove the following attainability of sharp constant $N_{Q,\alpha,p}$.

\begin{proposition}\label{sub reversed HLS exist pro}
(1)\ There exist a pair of nonnegative functions $(f,g)\in C^1(\mathbb{S}^{2n+1})\times C^1(\mathbb{S}^{2n+1})$ such that $\|f\|_{L^p(\mathbb{S}^{2n+1})}= \|g\|_{L^p(\mathbb{S}^{2n+1})}=1$ and
    $$N_{Q,\alpha,p}= \int_{\mathbb{S}^{2n+1}}\int_{\mathbb{S}^{2n+1}} \frac {f(\xi) g(\eta)} {|1-\xi\cdot\bar\eta|^{(Q-\alpha)/2}} d\xi d\eta.$$

(2) Minimizer pair $(f,g)$ satisfies the following Euler-Lagrange equations
\begin{equation}\label{EL sub Sn reversed}
    \begin{cases}
    N_{Q,\alpha,p} f^{p-1}(\xi) =\int_{\mathbb{S}^{2n+1}} |1-\xi\cdot\bar\eta|^{(\alpha-Q)/2} g(\eta) d\eta,\\
    N_{Q,\alpha,p} g^{p-1}(\xi) =\int_{\mathbb{S}^{2n+1}} |1-\xi\cdot\bar\eta|^{(\alpha-Q)/2} f(\eta) d\eta.
    \end{cases}
\end{equation}

(3)\ There exists some positive constant $C=C(Q,\alpha,p)$ such that
\begin{gather*}
    0<\frac 1C<f,g<C<+\infty,\\
\intertext{and}
    \|f\|_{C^1(\mathbb{S}^{2n+1})}, \|g\|_{C^1(\mathbb{S}^{2n+1})}\leq C.
\end{gather*}
\end{proposition}
\begin{proof}
We will divide the proof into three parts:

1.\ We show  that $N_{Q,\alpha,p}$ can be attained by a pair of nonnegative functions $(f,g)\in L^1(\mathbb{S}^{2n+1})\times L^1(\mathbb{S}^{2n+1})$.



By density argument, we can choose a pair of nonnegative minimizing sequence $\{f_j, g_j\}_{j=1}^{+\infty}\subset C^\infty(\mathbb{S}^{2n+1})\times C^\infty(\mathbb{S}^{2n+1})$ such that
    $$\|f_j\|_{L^{p}(\mathbb{S}^{2n+1})}= \|g_j\|_{L^{p}(\mathbb{S}^{2n+1})}=1,\ j=1,2,\cdots$$
and
    $$N_{Q,\alpha,p}=\lim_{j\rightarrow +\infty}\int_{\mathbb{S}^{2n+1}} \int_{\mathbb{S}^{2n+1}} f_j(\xi) g_j(\eta) |1-\xi\cdot\bar\eta|^{(\alpha-Q)/2} d\xi d\eta.$$

\textbf{ Step 1.}\ We prove that
\begin{equation}\label{sub re-3}
    \|f_j\|_{L^1(\mathbb{S}^{2n+1})}\leq C,\quad \|g_j\|_{L^1(\mathbb{S}^{2n+1})}\leq C,\quad \text{uniformly}.
\end{equation}
Indeed, from \eqref{sub re-0} we know that
there exist two constant $C_1$ and $C_2$ such that
$$0<C_1\le\int_{\mathbb{S}^{2n+1}} \int_{\mathbb{S}^{2n+1}} \frac{f_j(\xi) g_j(\eta)} {|1-\xi\cdot\bar\eta|^{(Q-\alpha)/2}} d\xi d\eta\le C_2<\infty.$$
By reversed H\"{o}lder's inequality, it holds that
\begin{eqnarray*}
\|I_{\alpha} f_j\|_{L^{p^\prime}(\mathbb{S}^{2n+1})}=\|g_j\|_{L^{p}(\mathbb{S}^{2n+1})}\|I_{\alpha} f_j\|_{L^{p^\prime}(\mathbb{S}^{2n+1})}&\le& C_2,\\
\|I_{\alpha} g_j\|_{L^{p^\prime}(\mathbb{S}^{2n+1})}=\|f_j\|_{L^{p}(\mathbb{S}^{2n+1})}\|I_{\alpha} g_j\|_{L^{p^\prime}(\mathbb{S}^{2n+1})}&\le& C_2,
\end{eqnarray*}
where $\frac{1}{p}+\frac{1}{p^\prime}=1$ and $I_\alpha f(\xi)=\int_{\mathbb{S}^{2n+1}} {|1-\xi\cdot\bar\eta|^{(Q-\alpha)/2}} f(\eta) d\eta$.
Noting that $0>p'>q_\alpha=\frac{2Q}{Q-\alpha}$, for some constant $M>0$ determined later, we have
\begin{eqnarray}\label{sub-re-s1}
C_2^{p'}&\le&\int_{\mathbb{S}^{2n+1}}|I_\alpha f_j|^{p'}d\xi=\int_{I_\alpha f_j\ge M}|I_\alpha f_j|^{p'}d\xi+\int_{I_\alpha f_j< M}|I_\alpha f_j|^{p'}d\xi\nonumber\\
&\le& M^{p'}|\mathbb{S}^{2n+1}|+\bigl |\{I_\alpha f_j< M\}\bigr |^{1-\frac{p'}{q_\alpha}} \big(\int_{I_\alpha f_j< M}|I_\alpha f_j|^{q_\alpha}d\xi\big)^\frac{p'}{q_\alpha}.
\end{eqnarray}
By reversed HLS inequality \eqref{re HLS roughly S} and reversed  H\"{o}lder inequality, we have
\begin{eqnarray}\label{sub-re-s2}
\|I_{\alpha} f_j\|_{L^{q_\alpha}(\mathbb{S}^{2n+1})}
&\ge& C_3\|f_j\|_{L^{p_\alpha}(\mathbb{S}^{2n+1})}\nonumber\\
&\ge& C_3 |\mathbb{S}^{2n+1}|^{\frac1{p_\alpha}-\frac1p}\|f_j\|_{L^{p}(\mathbb{S}^{2n+1})} =C_3|\mathbb{S}^{2n+1}|^{\frac1{p_\alpha}-\frac1p}.
\end{eqnarray}
We choose $M$ satisfying $ M^{p'}|\mathbb{S}^{2n+1}|=\frac12 C_2^{p'}$ and follow from  \eqref{sub-re-s1}  and \eqref{sub-re-s2} that
\begin{eqnarray*}
\frac12C_2^{p'}
&\le&\bigl |\{I_\alpha f_j< M\}\bigr |^{1-\frac{p'}{q_\alpha}} \big(\int_{|I_\alpha f_j|< M}|I_\alpha f_j|^{q_\alpha}d\xi\big)^\frac{p'}{q_\alpha}\\
&\le&\bigl |\{I_\alpha f_j< M\}\bigr |^{1-\frac{p'}{q_\alpha}} \big(\int_{\mathbb{S}^{2n+1}}|I_\alpha f_j|^{q_\alpha}d\xi\big)^\frac{p'}{q_\alpha}\\
&\le&\bigl |\{I_\alpha f_j< M\}\bigr |^{1-\frac{p'}{q_\alpha}} \big(C_3|\mathbb{S}^{2n+1}|^{\frac1{p_\alpha}-\frac1p}\big)^{p'},
\end{eqnarray*}
which leads to 
\begin{eqnarray*}
\bigl |\{I_\alpha f_j< M\}\bigr |\ge \big(\frac{C'_2}{C_3|\mathbb{S}^{2n+1}|^{\frac1{p_\alpha}-\frac1p}}\big)^{\frac{q_\alpha p'}{q_\alpha-p'}}>0,
\end{eqnarray*}
where $(c'_2)^{p'}=\frac12C_2^{p'}$. So, there exists  an $\epsilon_0>0$, such that for any $j\in \mathbb{N}^+$, we can find two points $\xi^1_j,\xi^2_j\in \{I_\alpha f_j< M\}$ satisfying $|\xi^1_j-\xi^2_j|\ge\epsilon_0$. Then
\begin{eqnarray*}
\int_{\mathbb{S}^{2n+1}} f_j(\xi)d\xi
&\le&\int_{\mathbb{S}^{2n+1}\backslash \{B(\xi^1_j,\frac{\epsilon_0}4)\}} f_j(\xi)d\xi
+\int_{\mathbb{S}^{2n+1}\backslash \{B(\xi^2_j,\frac{\epsilon_0}4)\}} f_j(\xi)d\xi\\
&\le&C_4\int_{\mathbb{S}^{2n+1}\backslash \{B(\xi^1_j,\frac{\epsilon_0}4)\}}|1-\xi^1_j\cdot\bar\eta|^{(Q-\alpha)/2} f_j(\xi)d\xi\\
& &+C_4\int_{\mathbb{S}^{2n+1}\backslash \{B(\xi^2_j,\frac{\epsilon_0}4)\}}|1-\xi^2_j\cdot\bar\eta|^{(Q-\alpha)/2} f_j(\xi)d\xi\\
&\le& 2C_4M.
\end{eqnarray*}
Hence, we obtain $ \|f_j\|_{L^1(\mathbb{S}^{2n+1})}\leq C$. In the same way, we have $\|g_j\|_{L^1(\mathbb{S}^{2n+1})}\leq C$.

\textbf{ Step 2.}\ There exist two subsequences of $\{f_j^p\}$ and $\{g_j^p\}$ (still denoted by $\{f_j^p\}$ and $\{g_j^p\}$) and two nonnegative functions $f, g\in L^{1}(\mathbb{S}^{2n+1})$ such that
\begin{equation}\label{sub re-1}
    \int_{\mathbb{S}^{2n+1}} f_j^p d\xi\rightarrow \int_{\mathbb{S}^{2n+1}} f^p d\xi,\quad \int_{\mathbb{S}^{2n+1}} g_j^p d\xi\rightarrow \int_{\mathbb{S}^{2n+1}} g^p d\xi,\quad \text{as}\quad j\rightarrow +\infty.
\end{equation}

In fact, according to the theory of  reflexive space, we know from \eqref{sub re-3} that there exist two subsequences of $\{f_j^p\}$ and $\{g_j^p\}$ (still denoted by $\{f_j^p\}$ and $\{g_j^p\}$) and two nonnegative functions $f, g\in L^{1}(\mathbb{S}^{2n+1})$ such that
    $$f_j^p\rightharpoonup f^p \quad \text{and}\quad \ g_j^p\rightharpoonup g^p \quad \text{weakly in}\quad L^{\frac1p}(\mathbb{S}^{2n+1}).$$
Using the fact $1\in L^{\frac1{1-p}}(\mathbb{S}^{2n+1})$,  we get \eqref{sub re-1}.

\textbf{ Step 3.}\ We show
\begin{equation}\label{sub re-2}
 \int_{\mathbb{S}^{2n+1}} \int_{\mathbb{S}^{2n+1}} \frac{f(\xi) g(\eta)} {|1-\xi\cdot\bar\eta|^{(Q-\alpha)/2}} d\xi d\eta\leq\liminf_{j\rightarrow +\infty}\int_{\mathbb{S}^{2n+1}} \int_{\mathbb{S}^{2n+1}} \frac{f_j(\xi) g_j(\eta)} {|1-\xi\cdot\bar\eta|^{(Q-\alpha)/2}} d\xi d\eta.
\end{equation}

As in Lemma 3.2 of  \cite{Dou-Guo-Zhu}, we have that, as $j\rightarrow +\infty$,
\begin{equation}\label{sub re-6}
 \int_{\mathbb{S}^{2n+1}}\frac{g_j^p(\eta) g^{1-p}(\eta)} {|1-\xi\cdot\bar\eta|^{(Q-\alpha)/2}} d\eta \rightarrow \int_{\mathbb{S}^{2n+1}}\frac{ g(\eta)} {|1-\xi\cdot\bar\eta|^{(Q-\alpha)/2}} d\eta
\end{equation}
uniformly for $\xi\in\mathbb{S}^{2n+1}$. Then, for any $\epsilon>0$, there exists some $N>0$ such that for any $j>N$,
    $$\Bigl|\int_{\mathbb{S}^{2n+1}}\frac{g_j^p(\eta) g^{1-p}(\eta)} {|1-\xi\cdot\bar\eta|^{(Q-\alpha)/2}} d\eta - \int_{\mathbb{S}^{2n+1}}\frac{ g(\eta)} {|1-\xi\cdot\bar\eta|^{(Q-\alpha)/2}} d\eta\Bigr|\leq \epsilon$$
and
\begin{align}\label{sub re-8}
    &\Bigl|\int_{\mathbb{S}^{2n+1}} f_j^p(\xi) f^{1-p}(\xi) \int_{\mathbb{S}^{2n+1}}\frac{g_j^p(\eta) g^{1-p}(\eta)} {|1-\xi\cdot\bar\eta|^{(Q-\alpha)/2}} d\eta d\xi\nonumber\\
    &\hspace{0.5cm}- \int_{\mathbb{S}^{2n+1}} f_j^p(\xi) f^{1-p}(\xi) \int_{\mathbb{S}^{2n+1}}\frac{ g(\eta)} {|1-\xi\cdot\bar\eta|^{(Q-\alpha)/2}} d\eta d\xi \Bigr|\nonumber\\
    &\leq \epsilon \int_{\mathbb{S}^{2n+1}} f_j^p(\xi) f^{1-p}(\xi) d\xi\leq C\epsilon.
\end{align}
On the other hand, noting $f^{1-p}(\xi)\in L^{1/(1-p)}(\mathbb{S}^{2n+1})$ and
    $$\int_{\mathbb{S}^{2n+1}} |1-\xi\cdot\bar\eta|^{(\alpha-Q)/2} g(\eta) d\eta\leq C\int_{\mathbb{S}^{2n+1}} g(\eta)d\eta\leq C,$$
we have by the weak convergence that, as $j\rightarrow +\infty$,
\begin{align}\label{sub re-7}
    &\int_{\mathbb{S}^{2n+1}} f_j^p(\xi) f^{1-p}(\xi) \int_{\mathbb{S}^{2n+1}}\frac{ g(\eta)} {|1-\xi\cdot\bar\eta|^{(Q-\alpha)/2}} d\eta d\xi\nonumber\\
    \rightarrow &\int_{\mathbb{S}^{2n+1}} \int_{\mathbb{S}^{2n+1}} \frac{f(\xi) g(\eta)} {|1-\xi\cdot\bar\eta|^{(Q-\alpha)/2}} d\eta d\xi.
\end{align}
Combining \eqref{sub re-8} and \eqref{sub re-7}, it holds that
\begin{align}\label{sub re-9}
    & \int_{\mathbb{S}^{2n+1}} \int_{\mathbb{S}^{2n+1}} \frac{f(\xi) g(\eta)} {|1-\xi\cdot\bar\eta|^{(Q-\alpha)/2}} d\eta d\xi\nonumber\\
    =&\lim_{j\rightarrow +\infty} \int_{\mathbb{S}^{2n+1}} \int_{\mathbb{S}^{2n+1}} \frac{f_j^p(\xi) f^{1-p}(\xi)g_j^p(\eta) g^{1-p}(\eta)} {|1-\xi\cdot\bar\eta|^{(Q-\alpha)/2}} d\eta d\xi \nonumber\\
    \leq &\lim_{j\rightarrow +\infty} \Bigl(\int_{\mathbb{S}^{2n+1}} \int_{\mathbb{S}^{2n+1}} \frac{f_j(\xi) g_j(\eta)} {|1-\xi\cdot\bar\eta|^{(Q-\alpha)/2}} d\eta d\xi\Bigr)^p\nonumber\\
    &\quad\cdot\Bigl(\int_{\mathbb{S}^{2n+1}} \int_{\mathbb{S}^{2n+1}} \frac{f(\xi) g(\eta)} {|1-\xi\cdot\bar\eta|^{(Q-\alpha)/2}} d\eta d\xi\Bigr)^{1-p}.
\end{align}
Thus, \eqref{sub re-2} holds.

Combining Step 1, Step 2 with Step 3, we know that the function pair $(f,g)\in L^1(\mathbb{S}^{2n+1})\times L^1(\mathbb{S}^{2n+1})$ is a minimizer.

2.\ We present that
$f,g$ satisfy the Euler-Lagrange equations \eqref{EL sub Sn reversed}.

Because $0<p<1$,  it brings some difficulties to deduce \eqref{EL sub Sn reversed}. To overcome it, we need to prove $f>0,\ g>0$ a.e. on $\mathbb{S}^{2n+1}$.

For any positive $\varphi\in C^\infty(\mathbb{S}^{2n+1})$ and $t>0$ small, we have $f+t\varphi>0$ on $\mathbb{S}^{2n+1}$ and
\begin{align}\label{sub re-10}
   & t\int_{\mathbb{S}^{2n+1}} \int_{\mathbb{S}^{2n+1}} \varphi(\xi) g(\eta) |1-\xi\cdot\bar\eta|^{(\alpha-Q)/2} d\xi d\eta \nonumber\\
  = & \int_{\mathbb{S}^{2n+1}} \int_{\mathbb{S}^{2n+1}} (f+t\varphi)(\xi) g(\eta) |1-\xi\cdot\bar\eta|^{(\alpha-Q)/2} d\xi d\eta\nonumber\\
  &-\int_{\mathbb{S}^{2n+1}} \int_{\mathbb{S}^{2n+1}} f(\xi) g(\eta) |1-\xi\cdot\bar\eta|^{(\alpha-Q)/2} d\xi d\eta\nonumber\\
  \geq & N_{Q,\alpha,p} \bigl(\|f+t\varphi\|_{L^p(\mathbb{S}^{2n+1})} -\|f\|_{L^p(\mathbb{S}^{2n+1})}\bigr)\nonumber\\
  =& N_{Q,\alpha,p} t\cdot \Bigl(\int_{\mathbb{S}^{2n+1}} (f+\theta\varphi)^p d\xi\Bigr)^{\frac 1p-1} \int_{\mathbb{S}^{2n+1}} (f+\theta\varphi)^{p-1}\varphi d\xi\ (0<\theta<t),
\end{align}
where the mean value theorem was used  between the fourth and fifth line. Then, by Fatou's lemma, it has
\begin{align}\label{sub re-10-1}
   & \int_{\mathbb{S}^{2n+1}} \int_{\mathbb{S}^{2n+1}} \varphi(\xi) g(\eta) |1-\xi\cdot\bar\eta|^{(\alpha-Q)/2} d\xi d\eta \nonumber\\
  \geq & N_{Q,\alpha,p} \lim_{t\rightarrow 0^+}\Bigl(\int_{\mathbb{S}^{2n+1}} (f+\theta\varphi)^p d\xi\Bigr)^{\frac 1p-1}\cdot \lim_{t\rightarrow 0^+}\int_{\mathbb{S}^{2n+1}} (f+\theta\varphi)^{p-1}\varphi d\xi \nonumber\\
  \geq & N_{Q,\alpha,p} \int_{\mathbb{S}^{2n+1}} f^{p-1}\varphi d\xi.
\end{align}
By now, we claim that $f>0$ a.e. on $\mathbb{S}^{2n+1}$. Otherwise, for any $\epsilon>0$, there exists $\Omega_\epsilon\subset\mathbb{S}^{2n+1}$ such that $|\Omega_\epsilon|>0$ and
    $$f(\xi)<\epsilon, \quad \forall\xi\in\Omega_\epsilon.$$
Then, it follows from \eqref{sub re-10-1} that
\begin{align*}
  \epsilon^{p-1}\int_{\Omega_\epsilon} d\xi \leq & \int_{\Omega_\epsilon} f^{p-1} d\xi \leq \frac 1{N_{Q,\alpha,p}}\int_{\mathbb{S}^{2n+1}} \int_{\mathbb{S}^{2n+1}} g(\eta) |1-\xi\cdot\bar\eta|^{(\alpha-Q)/2} d\xi d\eta\\
  \leq & C \int_{\mathbb{S}^{2n+1}} g(\eta) d\eta\leq C,
\end{align*}
which yields a contradiction as $\epsilon>0$ small enough. Similarly, we also have $g>0$ a.e. on $\mathbb{S}^{2n+1}$. So, minimizer pair $(f,g)$ is a weak solution of \eqref{EL sub Sn reversed}.

3.\ We finally prove that $(f,g)\in C^1(\mathbb{S}^{2n+1})\times C^1(\mathbb{S}^{2n+1})$. 

Since $f,g\in L^1(\mathbb{S}^{2n+1})$ and $0<p<p_\alpha<1$, it is easy to prove from \eqref{EL sub Sn reversed} that $f\geq C_6>0$ and $g\geq C_6>0$. Then, by \eqref{EL sub Sn reversed}, we have $f<C_7$ and $g<C_7$. Moreover, since $\alpha>Q\geq 4$, we have $f,g\in C^1(\mathbb{S}^{2n+1})$ and $\|f\|_{C^1(\mathbb{S}^{2n+1})}, \|f\|_{C^1(\mathbb{S}^{2n+1})}\leq C_8<+\infty$.
\end{proof}


\section{Sharp HLS inequalities on $\mathbb{S}^{2n+1}$}\label{sec critical}
\begin{lemma}\label{lem minimizing sequence}
$N_{Q,\alpha,p}\rightarrow N_{Q,\alpha}$ as $p\rightarrow p_\alpha^-$. Further more, the corresponding minimizer pairs $\{f_p,g_p\}\in C^1(\mathbb{S}^{2n+1})\times C^1(\mathbb{S}^{2n+1})$ 
form a minimizing sequence for sharp constant $N_{Q,\alpha}$, namely,
\begin{equation}\label{critical-2}
    N_{Q,\alpha}=\lim_{p\rightarrow p_\alpha^-} \frac{\int_{\mathbb{S}^{2n+1}} \int_{\mathbb{S}^{2n+1}} {f_p(\xi) g_p(\eta)} {|1-\xi\cdot\bar\eta|^{(\alpha-Q)/2}} d\xi d\eta} {\|f_p\|_{L^{p_\alpha}(\mathbb{S}^{2n+1})}  \|g_p\|_{L^{p_\alpha}(\mathbb{S}^{2n+1})}}.
\end{equation}
\end{lemma}
\begin{proof}
Let $\{f_p,g_p\}\in C^1(\mathbb{S}^{2n+1})\times C^1(\mathbb{S}^{2n+1})$ be  a pair of minimizer given by Proposition \ref{pro sub HLS}. Namely, $\{f_p,g_p\}$ satisfy $\|f_p\|_{L^p(\mathbb{S}^{2n+1})}= \|g_p\|_{L^p(\mathbb{S}^{2n+1})}=1$ and
    $$N_{Q,\alpha,p}= \int_{\mathbb{S}^{2n+1}}\int_{\mathbb{S}^{2n+1}} \frac {f_p(\xi) g_p(\eta)} {|1-\xi\cdot\bar\eta|^{(Q-\alpha)/2}} d\xi d\eta.$$
Write $\tilde{f}_p =\frac{f_p}{\|f_p\|_{L^{p_\alpha}(\mathbb{S}^{2n+1})}}$ and $\tilde{g}_p =\frac{g_p}{\|g_p\|_{L^{p_\alpha}(\mathbb{S}^{2n+1})}}$. Then
\begin{align*}
    N_{Q,\alpha,p}=& \|f_p\|_{L^{p_\alpha}(\mathbb{S}^{2n+1})} \|g_p\|_{L^{p_\alpha}(\mathbb{S}^{2n+1})} \int_{\mathbb{S}^{2n+1}} \int_{\mathbb{S}^{2n+1}} \frac {\tilde{f}_p(\xi) \tilde{g}_p(\eta)} {|1-\xi\cdot\bar\eta|^{(Q-\alpha)/2}} d\xi d\eta\\
    \geq & |\mathbb{S}^{2n+1}|^{2(1/p_\alpha-1/p)} N_{Q,\alpha} \rightarrow N_{Q,\alpha}, \quad \text{as} \quad p\rightarrow p_\alpha^-,
\end{align*}
which implies that
\begin{equation}\label{critical-4}
    \liminf_{p\rightarrow p_\alpha^+} N_{Q,\alpha,p}\geq N_{Q,\alpha}.
\end{equation}

Let $\{f_k,g_k\}_{k=1}^{+\infty}\subset L^{p_\alpha}(\mathbb{S}^{2n+1})\times L^{p_\alpha}(\mathbb{S}^{2n+1})$ be a pair of  minimizing sequence of $N_{Q,\alpha}$, namely,
    $$N_{Q,\alpha}=\lim_{k\rightarrow +\infty} \frac{\int_{\mathbb{S}^{2n+1}} \int_{\mathbb{S}^{2n+1}} f_k(\xi) g_k(\eta) |1-\xi\cdot\bar\eta|^{(\alpha-Q)/2} d\xi d\eta} {\|f_k\|_{L^{p_\alpha}(\mathbb{S}^{2n+1})} \|g_k\|_{L^{p_\alpha}(\mathbb{S}^{2n+1})}} .$$
Write $\tilde{f}_k =\frac{f_k}{\|f_k\|_{L^p(\mathbb{S}^{2n+1})}}$ and $\tilde{g}_k =\frac{g_k}{\|g_k\|_{L^p(\mathbb{S}^{2n+1})}}$ for any $p\in (0, p_\alpha)$. It is easy to see
\begin{align}\label{critical-1}
  N_{Q,\alpha,p}\leq &\frac{\int_{\mathbb{S}^{2n+1}} \int_{\mathbb{S}^{2n+1}} \tilde{f}_k(\xi) \tilde{g}_k(\eta) |1-\xi\cdot\bar\eta|^{(\alpha-Q)/2} d\xi d\eta} {\|\tilde{f}_k\|_{L^{p}(\mathbb{S}^{2n+1})} \|\tilde{g}_k\|_{L^{p}(\mathbb{S}^{2n+1})}}\nonumber\\
  =& \frac{\int_{\mathbb{S}^{2n+1}} \int_{\mathbb{S}^{2n+1}} f_k(\xi) g_k(\eta) |1-\xi\cdot\bar\eta|^{(\alpha-Q)/2} d\xi d\eta} {\|f_k\|_{L^{p}(\mathbb{S}^{2n+1})} \|g_k\|_{L^{p}(\mathbb{S}^{2n+1})}}.
\end{align}
Sending $p$ to $p_\alpha^-$ in \eqref{critical-1}, we get
    $$\limsup_{p\rightarrow p_\alpha^-} N_{Q,\alpha,p}\leq \frac{\int_{\mathbb{S}^{2n+1}} \int_{\mathbb{S}^{2n+1}} f_k(\xi) g_k(\eta) |1-\xi\cdot\bar\eta|^{(\alpha-Q)/2} d\xi d\eta} {\|f_k\|_{L^{p_\alpha}(\mathbb{S}^{2n+1})} \|g_k\|_{L^{p_\alpha}(\mathbb{S}^{2n+1})}}.$$
And then, letting $k\rightarrow +\infty$, we deduce
\begin{equation}\label{critical-5}
    \limsup_{p\rightarrow p_\alpha^+} N_{Q,\alpha,p}\leq N_{Q,\alpha}.
\end{equation}
Combining \eqref{critical-4} with \eqref{critical-5}, we arrive at $\lim_{p\rightarrow p_\alpha^+} N_{Q,\alpha,p}= N_{Q,\alpha}.$

By the definition of $N_{Q,\alpha}$ and H\"{o}lder inequality,
\begin{align*}
    N_{Q,\alpha}\leq &\frac{\int_{\mathbb{S}^{2n+1}} \int_{\mathbb{S}^{2n+1}} {f_p(\xi) g_p(\eta)} {|1-\xi\cdot\bar\eta|^{(\alpha-Q)/2}} d\xi d\eta} {\|f_p\|_{L^{p_\alpha}(\mathbb{S}^{2n+1})}  \|g_p\|_{L^{p_\alpha}(\mathbb{S}^{2n+1})}}\\
    \leq &\frac{\int_{\mathbb{S}^{2n+1}} \int_{\mathbb{S}^{2n+1}} {f_p(\xi) g_p(\eta)} {|1-\xi\cdot\bar\eta|^{(\alpha-Q)/2}} d\xi d\eta} {|\mathbb{S}^{2n+1}|^{2(1/p_\alpha-1/p)}}\\
    \rightarrow & N_{Q,\alpha}\quad \text{as}\quad p\rightarrow p_\alpha^-.
\end{align*}
Hence, we deduce that \eqref{critical-2} holds and the lemma is proved.
\end{proof}


\medskip

\textbf{Proof of Theorem \ref{pro HLS exist}.}\
As in Lemma \ref{lem minimizing sequence},
take the minimizer $\{f_p,g_p\}\in C^1(\mathbb{S}^{2n+1})\times C^1(\mathbb{S}^{2n+1})$ as a minimizing sequence for $N_{Q,\alpha}$. Then, $\{f_p,g_p\}$ satisfy \eqref{EL sub Sn reversed}. By the translation invariance, we assume, without loss of generality, that $f_p(\mathfrak{N})=\max_{\xi\in\mathbb{S}^n} f_{p}(\xi)$ with $\mathfrak{N}=(0,\cdots,0,1)$.

\textbf{Case 1:} For some subsequence $p_j\rightarrow p_\alpha^-$, $\max\{\max_{\xi\in\mathbb{S}^{2n+1}} f_{p_j}, \max_{\xi\in\mathbb{S}^{2n+1}} g_{p_j} \}$ is uniformly bounded. Then, sequences $\{f_{p_j}\}$ and $\{g_{p_j}\}$ are uniformly bounded and equicontinuous on $\mathbb{S}^{2n+1}$. Moreover, by \eqref{EL sub Sn reversed}, there exists some positive constant $C$ independent of $p_j$ such that $f_{p_j}, g_{p_j}\geq C>0$. So, by Arzel\`{a}-Ascoli theorem, there exist two subsequences of $\{f_{p_j}\}$ and $\{g_{p_j}\}$ (still denoted by $\{f_{p_j}\}$ and $\{g_{p_j}\}$) and two positive functions $f,g\in C^1(\mathbb{S}^{2n+1})$ such that
\begin{gather*}
  f_{p_j}\rightarrow f\quad \text{and} \quad g_{p_j}\rightarrow g \quad\text{ uniformly on }\quad \mathbb{S}^{2n+1}.
\end{gather*}
Then,
\begin{gather*}
    \int_{\mathbb{S}^{2n+1}} f^{p_\alpha}(\xi) d\xi=\lim_{p_j\rightarrow p_\alpha} \int_{\mathbb{S}^{2n+1}} f_{p_j}^{p_j}(\xi) d\xi=1,\\
    \int_{\mathbb{S}^{2n+1}} g^{p_\alpha}(\xi) d\xi=\lim_{p_j\rightarrow p_\alpha} \int_{\mathbb{S}^{2n+1}} g_{p_j}^{p_j}(\xi) d\xi=1.
\end{gather*}
Furthermore, by \eqref{EL sub Sn reversed} and Lemma \ref{lem minimizing sequence},
\begin{equation}\label{cri re-23}
    \begin{cases}
    N_{Q,\alpha}f^{p_\alpha-1}(\xi) =\int_{\mathbb{S}^{2n+1}} |1-\xi\cdot\bar\eta|^{(\alpha-Q)/2} g(\eta) d\eta,\\
    N_{Q,\alpha}g^{p_\alpha-1}(\xi) =\int_{\mathbb{S}^{2n+1}} |1-\xi\cdot\bar\eta|^{(\alpha-Q)/2} g(\eta) d\eta,
    \end{cases}
\end{equation}
 as $j\rightarrow +\infty$. Namely, $\{f,g\}$ are minimizers.

\textbf{Case 2:} For any subsequence $p_j\rightarrow p_\alpha^-$, $f_{p_j}(\mathfrak{N})\rightarrow +\infty$ or $\max_{\xi\in\mathbb{S}^n} g_{p_j} \rightarrow +\infty$. Without loss of generality, we assume $f_{p_j}(\mathfrak{N})\rightarrow +\infty$.

\textbf{Case 2a:} $\limsup_{j\rightarrow +\infty} \frac{f_{p_j}(\mathfrak{N})}{\max_{\xi\in\mathbb{S}^{2n+1}} g_{p_j}}=+\infty$. Then, there exists a subsequence of $\{p_j\}$ (still denoted by $\{p_j\}$) such that $f_{p_j}(\mathfrak{N})\rightarrow +\infty$ and $\frac{f_{p_j}(\mathfrak{N})}{\max_{\xi\in\mathbb{S}^{2n+1}} g_{p_j}}\rightarrow +\infty$. Let $\phi_j=f_{p_j}^{p_j-1}$ and $\psi_j=g_{p_j}^{p_j-1}$. Then, $\phi_j$ and $\psi_j$ satisfy
\begin{equation}\label{cri re_1}
  \int_{\mathbb{S}^{2n+1}} \phi_j^{q_j} d\xi=\int_{\mathbb{S}^{2n+1}} \psi_j^{q_j} d\xi=1
\end{equation}
and by \eqref{EL sub Sn reversed},
\begin{equation}\label{cri re-2}
    \begin{cases}
      N_{Q,\alpha,p_j} \phi_j(\xi)=\int_{\mathbb{S}^{2n+1}} |1-\xi\cdot\bar\eta|^{(\alpha-Q)/2} \psi_{j}^{q_j-1}(\eta) d\eta,\\
      N_{Q,\alpha,p_j} \psi_j(\xi)=\int_{\mathbb{S}^{2n+1}} |1-\xi\cdot\bar\eta|^{(\alpha-Q)/2} \phi_j^{q_j-1}(\eta) d\eta,
    \end{cases}
\end{equation}
where $\frac{1}{p_j}+\frac{1}{q_j}=1$.
Applying Cayley transformation and dilations on $\mathbb{H}^n$, we get from \eqref{cri re-2} that
\begin{equation}\label{cri re-3}
    \begin{cases}
        \frac{N_{Q,\alpha,p_j} \phi_j(\mathcal{C}(\delta_\lambda (u)))}{\left((1+|\lambda z|^2)^2+(\lambda^2 t)^2\right)^{\frac{Q-\alpha}{4}}} =2^{\frac{Q+\alpha-2}2} \lambda^\alpha \int_{\mathbb{H}^n} \frac{\left((1+|\lambda z'|^2)^2+(\lambda^2 t')^2\right)^{-\frac{Q+\alpha}{4}} \psi_j(\mathcal{C}(\delta_\lambda (v)))^{q_j-1}} {|u^{-1}v|^{Q-\alpha}} dv,\\
        \frac{N_{Q,\alpha,p_j} \psi_j(\mathcal{C}(\delta_\lambda (u)))}{\left((1+|\lambda z|^2)^2+(\lambda^2 t)^2\right)^{\frac{Q-\alpha}{4}}} =2^{\frac{Q+\alpha-2}2} \lambda^\alpha \int_{\mathbb{H}^n} \frac{\left((1+|\lambda z'|^2)^2+(\lambda^2 t')^2\right)^{-\frac{Q+\alpha}{4}} \phi_j(\mathcal{C}(\delta_\lambda (v)))^{q_j-1}} {|u^{-1}v|^{Q-\alpha}} dv.
    \end{cases}
\end{equation}
Take $\lambda=\lambda_j$ satisfying $\lambda_j^{\alpha/(q_j-2)} \phi_j(\mathcal{C}(0))=1$ and denote
\begin{equation}\label{cri re-14}\begin{cases}
  \Phi_j(u)=\frac{ \lambda_j^{\alpha/(q_j-2)} } {\left((1+|\lambda z|^2)^2+(\lambda^2 t)^2\right)^{\frac{Q-\alpha}{4}}} \phi_j(\mathcal{C}(\delta_{\lambda_j} (u))),\\
  \Psi_j(u)=\frac{ \lambda_j^{\alpha/(q_j-2)} } {\left((1+|\lambda z|^2)^2+(\lambda^2 t)^2\right)^{\frac{Q-\alpha}{4}}} \psi_j(\mathcal{C}(\delta_{\lambda_j} (u))).
\end{cases}\end{equation}
Then, $\Phi_j,\Psi_j$ satisfy the following renormalized equations
\begin{equation}\label{cri re-5}
  \begin{cases}
    N_{Q,\alpha,p_j} \Phi_j(u)=2^{\frac{Q+\alpha-2}2} \int_{\mathbb{H}^n} \left((1+|\lambda_j z'|^2)^2+(\lambda_j^2 t')^2\right)^{\frac{\alpha-Q}{4}(q_\alpha-q_j)} \frac{\Psi_j^{q_j-1}(v)}{|u^{-1}v|^{Q-\alpha}} dv, \\
    N_{Q,\alpha,p_j} \Psi_j(u)=2^{\frac{Q+\alpha-2}2} \int_{\mathbb{H}^n} \left((1+|\lambda_j z'|^2)^2+(\lambda_j^2 t')^2\right)^{\frac{\alpha-Q}{4}(q_\alpha-q_j)} \frac{\Phi_j^{q_j-1}(v)}{|u^{-1}v|^{Q-\alpha}} dv.
  \end{cases}
\end{equation}
Moreover, $\Phi_j(u)\geq \Phi_j(0)=1$ and
\begin{equation}\label{cri re-7}
    \Psi_j(u)\geq \lambda_j^{\alpha/(q_j-2)} \min_{\xi\in \mathbb{S}^{2n+1}} \psi_j=\frac{\min_{\xi\in \mathbb{S}^{2n+1}} \psi_j}{\phi_j(\mathcal{S}(0))}\rightarrow +\infty
\end{equation}
uniformly for any $u$ as $j\rightarrow +\infty$.

Claim:  There exist $C_1, C_2>0$ such that, for any $u\in\mathbb{H}^n$, when $j\to\infty$,
\begin{equation}\label{cri re-4}
0<C_1(1+|u|^{\alpha-Q})\le \Phi_j(u) \le C_2(1+|u|^{\alpha-Q}) \ \mbox{uniformly.}
\end{equation}
Once the claim holds,
\begin{align*}
    N_{Q,\alpha,p_j}\Psi_j(0)=& 2^{\frac{Q+\alpha-2}2} \int_{\mathbb{H}^n} \left((1+|\lambda_j z'|^2)^2+(\lambda_j^2 t')^2\right)^{\frac{\alpha-Q}{4}(q_\alpha-q_j)} \frac{\Phi_j^{q_j-1}(v)}{|v|^{Q-\alpha}} dv\\
    \leq & C\int_{\mathbb{H}^n} |v|^{\alpha-Q} (1+|v|^{\alpha-Q})^{q_j-1} dv\leq C,
\end{align*}
which contradicts with \eqref{cri re-7}. This shows that {\bf Case 2a} does not appear.

Now, we give the proof of the claim \eqref{cri re-4}.
Noting that
\begin{equation}\label{cri re-6}
    N_{Q,\alpha,p_j} 
    =2^{\frac{Q+\alpha-2}2} \int_{\mathbb{H}^n} \left((1+|\lambda_j z'|^2)^2+(\lambda_j^2 t')^2\right)^{\frac{\alpha-Q}{4}(q_\alpha-q_j)} \frac{\Psi_j^{q_j-1}(v)}{|v|^{Q-\alpha}} dv\leq C<+\infty
\end{equation}
uniformly as $j\to\infty$, we obtain from \eqref{cri re-7} that as $j\to\infty$ and $|u|\geq 1$,
\begin{align}\label{cri re-8}
    &\int_{\mathbb{H}^n} \left((1+|\lambda_j z'|^2)^2+(\lambda_j^2 t')^2\right)^{\frac{\alpha-Q}{4}(q_\alpha-q_j)} \Psi_j^{q_j-1}(v) dv\nonumber\\
    \leq &\int_{|v|\leq 1} C \left((1+|\lambda_j z'|^2)^2+(\lambda_j^2 t')^2\right)^{\frac{\alpha-Q}{4}(q_\alpha-q_j)} dv\nonumber\\ &+\int_{|v|>1} \left((1+|\lambda_j z'|^2)^2+(\lambda_j^2 t')^2\right)^{\frac{\alpha-Q}{4}(q_\alpha-q_j)} \frac{\Psi_j^{q_j-1}(v)}{|v|^{Q-\alpha}} dv\leq C<+\infty
\end{align}
and
\begin{align}\label{cri re-9}
    &\int_{\mathbb{H}^n} \frac{|u^{-1}v|^{\alpha-Q}}{|u|^{\alpha-Q}} \left((1+|\lambda_j z'|^2)^2+(\lambda_j^2 t')^2\right)^{\frac{\alpha-Q}{4}(q_\alpha-q_j)} \Psi_j^{q_j-1}(v) dv\nonumber\\
    \leq &C\int_{\mathbb{H}^n} (1+|v|^{\alpha-Q}) \left((1+|\lambda_j z'|^2)^2+(\lambda_j^2 t')^2\right)^{\frac{\alpha-Q}{4}(q_\alpha-q_j)} \Psi_j^{q_j-1}(v) dv\nonumber\\
    \leq & C<+\infty
\end{align}
uniformly. By dominated convergence theorem,
\begin{align}\label{cri re-10}
    &\lim_{|u|\rightarrow +\infty} \frac{\Phi_j(u)}{|u|^{\alpha-Q}}\nonumber\\
    =& \frac{2^{\frac{Q+\alpha-2}2}} {N_{Q,\alpha, p_j}} \int_{\mathbb{H}^n} \left((1+|\lambda_j z'|^2)^2+(\lambda_j^2 t')^2\right)^{\frac{\alpha-Q}{4}(q_\alpha-q_j)} \Psi_j^{q_j-1}(v) dv \leq C.
\end{align}
On the other hand, if we can prove
\begin{equation}\label{cri re-11}
    \int_{\mathbb{H}^n} \left((1+|\lambda_j z'|^2)^2+(\lambda_j^2 t')^2\right)^{\frac{\alpha-Q}{4}(q_\alpha-q_j)} \Psi_j^{q_j-1}(v) dv\ge C'>0 \ \ \mbox{as} \ \ j\to\infty,
\end{equation}
then we have the claim \eqref{cri re-4}. By contradiction, we assume that \eqref{cri re-11} does not hold. Then, there exists a subsequence (still denoted as $\{\Psi_j\}$) such that
\begin{equation}\label{cri re-13}
    \int_{\mathbb{H}^n} \left((1+|\lambda_j z'|^2)^2+(\lambda_j^2 t')^2\right)^{\frac{\alpha-Q}{4}(q_\alpha-q_j)} \Psi_j^{q_j-1}(v) dv\rightarrow 0, \quad\text{as}\quad j\rightarrow+\infty.
\end{equation}
For any $u\in B(0, 1)$,
\begin{align*}
    1\le & \Phi_j(u) =\frac{2^{\frac{Q+\alpha-2}2}} {N_{Q,\alpha, p_j}} \int_{\mathbb{H}^n} \left((1+|\lambda_j z'|^2)^2+(\lambda_j^2 t')^2\right)^{\frac{\alpha-Q}{4}(q_\alpha-q_j)} \frac{\Psi_j^{q_j-1}(v)}{|u^{-1}v|^{Q-\alpha}} dv\\
    \le&\frac{2^{\frac{Q+\alpha-2}2}} {N_{Q,\alpha, p_j}} \Bigl( (4\gamma)^{\alpha-Q} \int_{|v|\leq 3} \left((1+|\lambda_j z'|^2)^2+(\lambda_j^2 t')^2\right)^{\frac{\alpha-Q}{4}(q_\alpha-q_j)} \Psi_j^{q_j-1}(v) dv\nonumber\\
    &\hspace{0.5cm}+\int_{|v|>3} (\frac 43\gamma)^{\alpha-Q} \left((1+|\lambda_j z'|^2)^2+(\lambda_j^2 t')^2\right)^{\frac{\alpha-Q}{4}(q_\alpha-q_j)} \frac{\Psi_j^{q_j-1}(v)}{|v|^{Q-\alpha}} dv\Bigr)\\
    \le&\frac{2^{\frac{Q+\alpha-2}2}} {N_{Q,\alpha, p_j}} (4\gamma)^{\alpha-Q} \int_{|v|\leq 3} \left((1+|\lambda_j z'|^2)^2+(\lambda_j^2 t')^2\right)^{\frac{\alpha-Q}{4}(q_\alpha-q_j)} \Psi_j^{q_j-1}(v) dv\\    &\hspace{0.2cm}+ (\frac 43\gamma)^{\alpha-Q}.
\end{align*}
From \eqref{cri re-13}, there exists $N_0>0$ such that $$1\leq\Phi_j(u)\leq 1+(\frac 43\gamma)^{\alpha-Q}$$
for $j\geq N_0$. Then, for $|u|\geq 3$ , if follows from \eqref{cri re-5}  that
\begin{align}\label{cri re-12}
    \Psi_j(u)\geq & \frac{2^{\frac{Q+\alpha-2}2}} {N_{Q,\alpha, p_j}}  \int_{|v|\leq 1} \left((1+|\lambda_j z'|^2)^2+(\lambda_j^2 t')^2\right)^{\frac{\alpha-Q}{4}(q_\alpha-q_j)} \frac{\Phi_j^{q_j-1}(v)}{|u^{-1}v|^{Q-\alpha}} dv\nonumber\\
    \geq & C \int _{|v|\leq 1} {|u|}^{\alpha-Q} \left((1+|\lambda_j z'|^2)^2+(\lambda_j^2 t')^2\right)^{\frac{\alpha-Q}{4}(q_\alpha-q_j)} dv
    \geq  C|u|^{\alpha-Q},
\end{align}
for $j\geq N_0$. We used the fact  in the last inequality: as $j\to \infty$,
    $$\left((1+|\lambda_j z'|^2)^2+(\lambda_j^2 t')^2\right)^{\frac{\alpha-Q}{4}(q_\alpha-q_j)} \rightarrow 1\quad \text{uniformly on}\ B(0,1).$$
Letting $p_j$ close to $p_\alpha$ and choosing $R>>3$, it follows from \eqref{cri re-12}  that
\begin{align*}
    N_{Q,\alpha,p_j}=& 2^{\frac{Q+\alpha-2}2} \int_{\mathbb{H}^n} \left((1+|\lambda_j z'|^2)^2+(\lambda_j^2 t')^2\right)^{\frac{\alpha-Q}{4}(q_\alpha-q_j)} \frac{\Psi_j^{q_j-1}(v)}{|v|^{Q-\alpha}} dv\\
    \leq & C R^{\alpha-Q}\int_{|v|\leq R} \left((1+|\lambda_j z'|^2)^2+(\lambda_j^2 t')^2\right)^{\frac{\alpha-Q}{4}(q_\alpha-q_j)} \Psi_j^{q_j-1}(v)  dv\\
    & +C\int_{|v|>R}|v|^{\alpha-Q} \cdot |v|^{(\alpha-Q)(q_j-1)} dy\\
    \leq &  C R^{\alpha-Q}\int_{|v|\leq R} \left((1+|\lambda_j z'|^2)^2+(\lambda_j^2 t')^2\right)^{\frac{\alpha-Q}{4}(q_\alpha-q_j)} \Psi_j^{q_j-1}(v)  dv\\
    &\hspace{0.5cm}+ C R^{(\alpha-Q)q_j+Q}.
\end{align*}
Taking firstly $R$ large enough and then letting $j\rightarrow +\infty$, we have $N_{Q,\alpha,p_j}\rightarrow 0$, which is contradiction with $N_{Q,\alpha,p_j}\rightarrow N_{Q,\alpha}$. Hence, \eqref{cri re-11} holds.

\textbf{Case 2b:} $\limsup_{j\rightarrow +\infty} \frac{f_{p_j}(\mathfrak{N})}{\max_{\xi\in\mathbb{S}^{2n+1}} g_{p_j}}=0$. Then, there exists a subsequence of $\{p_j\}$ (still denoted as $\{p_j\}$) such that $f_{p_j}(\mathfrak{N})\rightarrow +\infty$ and $\frac{f_{p_j}(\mathfrak{N})}{\max_{\xi\in\mathbb{S}^{2n+1}} g_{p_j}}\rightarrow 0$, which implies that $\max_{\xi\in\mathbb{S}^{2n+1}} g_{p_j}\rightarrow +\infty$. Similar to {\bf Case 2a}, we can show that {\bf Case 2b} does not appear.

\textbf{Case 2c:} $\limsup_{j\rightarrow +\infty} \frac{f_{p_j}(\mathfrak{N})}{\max_{\xi\in\mathbb{S}^{2n+1}} g_{p_j}}=c_0\in (0,+\infty)$. Then, there exists a  subsequence of $\{p_j\}$ (still denoted as $\{p_j\}$) such that $f_{p_j}(\mathfrak{N})\rightarrow +\infty$, $\max g_{p_j}\rightarrow +\infty$ and $\frac{f_{p_j}(\mathfrak{N})}{\max_{\xi\in\mathbb{S}^{2n+1}} g_{p_j}} \rightarrow c_0\in (0,+\infty)$. As {\bf Case 2a}, choose a sequence of function pairs $\{\Phi_j, \Psi_j\}$ defined as \eqref{cri re-14}, which satisfies \eqref{cri re-5}, $\Phi_j(u)\geq \Phi_j(0)=1$ and
\begin{equation}\label{cri re-15}
    \Psi_j(u)\geq \lambda_j^{\alpha/(q_j-2)} \min_{\xi\in\mathbb{S}^{2n+1}} \psi_j=\frac{\min_{\xi\in\mathbb{S}^{2n+1}} \psi_j}{\phi_j(\mathcal{C}(0))}\rightarrow c_0^{1-p_\alpha} \in (0,+\infty)
\end{equation}
uniformly for any $u$ as $j\rightarrow +\infty$. So, $\{\Psi_j(u)\}$ have uniformly lower bound $C>0$.

Repeating the proof of \eqref{cri re-4}, there exist two positive constants $C_1$ and $C_2$ such that, as $j\rightarrow +\infty$,
\begin{gather}
    0<C_1(1+|u|^{\alpha-Q})\le \Phi_j(u) \le C_2(1+|u|^{\alpha-Q}),\label{cri re-16}\\
    0<C_1(1+|u|^{\alpha-Q})\le \Psi_j(u) \le C_2(1+|u|^{\alpha-Q})\label{cri re-17}
\end{gather}
uniformly for any $u$.

For any given constant $R_0>0$ and any $u\in B(0,R_0)$, as $j\rightarrow +\infty$, we have by \eqref{cri re-17} that
\begin{align}\label{cri re-18}
    &2^{-\frac{Q+\alpha-2}2}N_{Q,\alpha,p_j}\Phi_j(u)= \int_{\mathbb{H}^n} \left((1+|\lambda_j z'|^2)^2+(\lambda_j^2 t')^2\right)^{\frac{\alpha-Q}{4}(q_\alpha-q_j)} \frac{\Psi_j^{q_j-1}(v)}{|u^{-1}v|^{Q-\alpha}} dv \nonumber\\
    =& \int_{\mathbb{H}^n} \left((1+|\lambda_j (z+z')|^2)^2+(\lambda_j^2 (t+t'+2\text{Im}(z\cdot\bar{z'})))^2\right)^{\frac{\alpha-Q}{4}(q_\alpha-q_j)} \frac{\Psi_j^{q_j-1}(uv)}{|v|^{Q-\alpha}} dv \nonumber\\
    \leq & \int_{|v|\leq 2R_0} |v|^{\alpha-Q} C_1^{q_j-1} dv +\int_{|v|>2R_0} |v|^{\alpha-Q}  C_1^{q_j-1}|uv|^{(\alpha-Q)(q_j-1)}dv \nonumber\\
    \leq & C(2R_0)^\alpha +C\int_{|v|>2R_0} |v|^{(\alpha-Q)q_j}dv\leq C,
\end{align}
namely, $\Phi_j(u)$ is uniformly bounded on $B(0,R_0)$. Similarly, $\Psi_j(u)$ is also uniformly bounded on $B(0,R_0)$.

Noting $\alpha>Q\geq 4$ and arguing as \eqref{cri re-18}, we have that, as $j\rightarrow+\infty$,
\begin{align*}
    &\int_{\mathbb{H}^n} \left((1+|\lambda_j z'|^2)^2+(\lambda_j^2 t')^2\right)^{\frac{\alpha-Q}{4}(q_\alpha-q_j)} \frac{\Psi_j^{q_j-1}(v)}{|u^{-1}v|^{Q-\alpha+1}} dv\leq C
\intertext{and}
    &\int_{\mathbb{H}^n} \left((1+|\lambda_j z'|^2)^2+(\lambda_j^2 t')^2\right)^{\frac{\alpha-Q}{4}(q_\alpha-q_j)} \frac{\Psi_j^{q_j-1}(v)}{|u^{-1}v|^{Q-\alpha+2}} dv\leq C
\end{align*}
uniformly for any $u\in B(0,R_0)$. So, for any $u\in B(0,R_0)$ and $l,k=1,2,\cdots,2n$, a direct computation yields
\begin{align*}
    &T_lT_k \Phi_j(u)= \frac{2^{\frac{Q+\alpha-2}2}}{N(Q,\alpha,p_j)} \int_{\mathbb{H}^n} T_lT_k (|u^{-1}v|^{\alpha-Q}) \frac{\Psi_j^{q_j-1}(v)}{\left((1+|\lambda_j z'|^2)^2+(\lambda_j^2 t')^2\right)^{\frac{Q-\alpha}{4}(q_\alpha-q_j)}} dv,
\end{align*}
where $T_l=X_l, T_{l+n}=Y_l$ for $l=1,2,\cdots,n$ and 
    $$ X_l=\frac{\partial}{\partial{x_i}}+2y_i\frac{\partial}{\partial{t}},\quad Y_l
    =\frac{\partial}{\partial{y_i}}-2x_i\frac{\partial}{\partial{t}},\quad l=1,2,\cdots,n$$ 
are the
left invariant vector fields on the Heisenberg group. Moreover, we know that $\Phi_j\in C^1(B(0,R_0))$ by Theorem 20.1 of \cite{Folland-Stein1974}.
Since the arbitrariness of $R_0$, we know that $\Phi_j(u)\in C^1(\mathbb{H}^n)$ and $\|\Phi_j\|_{C^1(B(0,R_0))}$ is uniformly bounded. Similarly, we can obtain that $\Psi_j(u)\in C^1(\mathbb{H}^n)$ and $\|\Psi_j\|_{C^1(B(0,R_0))}$ is uniformly bounded.

By Arzel\`{a}-Ascoli theorem, there exist two subsequences of $\{\Phi_j\}$ and $\{\Psi_j\}$ (still denoted as $\{\Phi_j\}$ and $\{\Psi_j\}$) and two functions $U,V\in C^1(\mathbb{H}^n)$ such that
\begin{equation}\label{cri re-19}
    \Phi_j\rightarrow U\quad \text{and}\quad \Psi_j\rightarrow V \quad\text{uniformly on} \quad B(0,R_0).
\end{equation}
Moreover, by \eqref{cri re-16} and \eqref{cri re-17}, it holds
\begin{gather}
    0<C_1(1+|u|^{\alpha-Q})\le U(u) \le C_2(1+|u|^{\alpha-Q}),\label{cri re-21}\\
    0<C_1(1+|u|^{\alpha-Q})\le V(u) \le C_2(1+|u|^{\alpha-Q}).\label{cri re-22}
\end{gather}
By the arbitrariness of $R_0$, we  prove that $U(u)$ and $V(u)$ satisfy
\begin{equation}\label{cri re-20}
    \begin{cases} N_{Q,\alpha} U(u)=2^{\frac{Q+\alpha-2}2}\int_{\mathbb{H}^n} |u^{-1}v|^{\alpha-Q} V^{q_\alpha-1}(v)dv\quad \text{in}\quad \mathbb{H}^n,\\ N_{Q,\alpha} V(u)=2^{\frac{Q+\alpha-2}2} \int_{\mathbb{H}^n} |u^{-1}v|^{\alpha-Q} U^{q_\alpha-1}(v)dv\quad \text{in}\quad \mathbb{H}^n.\end{cases}
\end{equation}
Since
\begin{align*}
    1 =&\int_{\mathbb{S}^{2n+1}} \phi_j^{q_j}(\xi)d\xi \nonumber\\
    =& 2^{Q-1} \int_{\mathbb{H}^n} \Phi_j^{q_j}(u) \lambda_j^{Q-\frac{\alpha q_j}{q_j-2}} {\left((1+|\lambda_j z'|^2)^2+(\lambda_j^2 t')^2\right)^{\frac{\alpha-Q}{4}(q_\alpha-q_j)}} dv\\
    \leq & 2^{Q-1} \int_{\mathbb{H}^n} \Phi_j^{q_j}(u) {\left((1+|\lambda_j z'|^2)^2+(\lambda_j^2 t')^2\right)^{\frac{\alpha-Q}{4}(q_\alpha-q_j)}} dv
\end{align*}
and
$$\Phi_j^{q_j}(u) \left((1+|\lambda_j z'|^2)^2+(\lambda_j^2 t')^2\right)^{\frac{\alpha-Q}{4}(q_\alpha-q_j)} \rightarrow U^{q_\alpha}(u)$$ uniformly on any compact domain,  it follows from \eqref{cri re-16} that
    $$\int_{\mathbb{H}^n}U^{q_\alpha}du =\lim_{j\rightarrow +\infty} \int_{\mathbb{H}^n} \Phi_j^{q_j}(u) \left((1+|\lambda_j z'|^2)^2+(\lambda_j^2 t')^2\right)^{\frac{\alpha-Q}{4}(q_\alpha-q_j)} du \geq 2^{1-Q}.$$
Similarly, by \eqref{cri re-17} it also holds
    $$\int_{\mathbb{H}^n}V^{q_\alpha}du =\lim_{j\rightarrow +\infty} \int_{\mathbb{H}^n} \Psi_j^{q_j}(u) \left((1+|\lambda_j z'|^2)^2+(\lambda_j^2 t')^2\right)^{\frac{\alpha-Q}{4}(q_\alpha-q_j)} du \geq 2^{1-Q}.$$
Let $F(u)=U^{q_\alpha-1}(u)$ and $G(u)=V^{q_\alpha-1}(u)$, we have $\int_{\mathbb{H}^n} F^{p_\alpha}du\geq 2^{1-Q}$, $\int_{\mathbb{H}^n} G^{p_\alpha} du\geq 2^{1-Q}$ and $F,G$ satisfy
    $$\begin{cases} N_{Q,\alpha} F^{p_\alpha-1}(u)=2^{\frac{Q+\alpha-2}2} \int_{\mathbb{H}^n} |u^{-1}v|^{\alpha-Q} G(v)dv\quad \text{in}\quad \mathbb{H}^n,\\ N_{Q,\alpha} G^{p_\alpha-1}(u)=2^{\frac{Q+\alpha-2}2} \int_{\mathbb{H}^n} |u^{-1}v|^{\alpha-Q} F(v)dv\quad \text{in}\quad \mathbb{H}^n.\end{cases}$$
Combining $2>p_\alpha$  wtih Cayley transformation, it holds
\begin{align*}
    N_{Q,\alpha}^2=& \frac{\left(2^{\frac{Q+\alpha-2}2} \int_{\mathbb{H}^n} \int_{\mathbb{H}^n} F(u) |u^{-1}v|^{\alpha-Q} G(v) dv du\right)^2}{\int_{\mathbb{H}^n} F^{p_\alpha} du \int_{\mathbb{H}^n} G^{p_\alpha} du}\\
    \geq & \frac{\left(2^{\frac{Q+\alpha-2}2} \int_{\mathbb{H}^n} \int_{\mathbb{H}^n} F(u) |u^{-1}v|^{\alpha-Q} G(v) dv du\right)^2}{2^{2-2Q} \left(2^{Q-1}\int_{\mathbb{H}^n} F^{p_\alpha} du\right)^{2/p_\alpha} \left(2^{Q-1} \int_{\mathbb{H}^n} G^{p_\alpha} du\right)^{2/p_\alpha}}\\
    =& \frac{\left(2^{\frac{Q+\alpha-2}2} 2^{-\frac{n(Q-\alpha)}Q} \int_{\mathbb{S}^{2n+1}} \int_{\mathbb{S}^{2n+1}} f(\xi) |1-\xi\cdot\bar\eta|^{(\alpha-Q)/2} g(\eta) d\eta d\xi\right)^2}{2^{2-2Q} 2^{\frac{4(Q-1)}{p_\alpha}} \left(\int_{\mathbb{S}^{2n+1}} f^{p_\alpha} d\xi\right)^{2/p_\alpha} \left( \int_{\mathbb{S}^{2n+1}} g^{p_\alpha} d\xi\right)^{2/p_\alpha}}\\
    =& \frac{\left( \int_{\mathbb{S}^{2n+1}} \int_{\mathbb{S}^{2n+1}} f(\xi) |1-\xi\cdot\bar\eta|^{(\alpha-Q)/2} g(\eta) d\eta d\xi\right)^2}{ \left(\int_{\mathbb{S}^{2n+1}} f^{p_\alpha} d\xi\right)^{2/p_\alpha} \left( \int_{\mathbb{S}^{2n+1}} g^{p_\alpha} d\xi\right)^{2/p_\alpha}},
\end{align*}
where
\begin{gather*}
    F(u)=f(\mathcal{C}(u)) J_\mathcal{C}(u)^{1/p_\alpha},\quad
    G(u)=g(\mathcal{C}(u)) J_\mathcal{C}(u)^{1/p_\alpha}.
\end{gather*}
Hence, $\{f(\xi),g(\xi)\}$ is a pair of minimizer of sharp constant $N_{Q,\alpha}$. Furthermore, they satisfy the Euler-Lagrange equations \eqref{cri re-23}.

By \eqref{cri re-21} and \eqref{cri re-22}, there exists a positive constant $C$ such that
    $$0<\frac 1C\leq f,g\leq C.$$
Since $\alpha>Q\geq 4$, we know by \eqref{cri re-23} that $f,g\in C^1(\mathbb{S}^{2n+1})$. \hspace*{\stretch{1}}$\Box$


\medskip

\noindent {\bf Acknowledgements}\\
\indent The author would like to thank Professor Meijun Zhu for valuable discussions and suggestions. The project is supported by  the National Natural Science Foundation of China (Grant No. 12071269) and Natural Science Foundation of Zhejiang Province (Grant No.  LY18A010013). 


\small
\begin{thebibliography}{a}\small

\bibitem{B1993} W. Beckner, Sharp Sobolev inequalities on the sphere and the Moser-Trudinger inequality, Ann. of Math., 138 (1993), 213-242.

\bibitem{B2015} W. Beckner, Functionals for Multilinear Fractional Embedding, Acta Math. Sinica,
English Series, 31 (2015), 1-28.

\bibitem{BFM2013} T.P. Branson, L. Fontana, C. Morpurgo, Moser-Trudinger and Beckner-Onofri's inequalities on the CR sphere, Ann. of Math., 177(2013), 1-52.

\bibitem{BP1999} I. Birindelli, J. Prajapat, Nonlinear Liouville theorems in the Heisenberg group via the moving plane method, Comm. PDE., 24(9\&10)(1999), 1875-1890.

\bibitem{Carlen-Loss} E.A. Carlen, M. Loss, Extremals of functionals with competing symmmetries, J. Funct. Anal. 88(2)(1990), 437-456.
\bibitem{CLT2019} L. Chen, G. Lu, C. Tao, Existence of extremal functions for the Stein-Weiss inequalities on the Heisenberg group, J. Func. Anal., 277(2019), 1112-1138.

\bibitem{CLO2006} W. Chen, C. Li, B. Ou, Classification of solutions for an integral equation, Comm. Pure Appl. Math.,  59 (2006), 330-343.

\bibitem{Cohn-Lu2001} W. Cohn, G. Lu, Best constants for Moser-Trudinger inequalities on the Heisenberg group, Indiana Univ. Math. J., 50(4)(2001), 1567-1591.

\bibitem{Cohn-Lu2004} W. Cohn, G. Lu, Sharp constants for Moser-Trudinger inequalities on spheres in complex space $\mathbb{C}^n$, Comm. Pure Appl. Math.,  57(2004), 1458-1493.


\bibitem{Dou-Guo-Zhu} J. Dou, Q. Guo, M. Zhu, Negative power nonlinear integral equations on bounded domains, J. Diff. Equ., 269(2020), 10527-10557.

\bibitem{Dou-Guo-Zhu2017}J. Dou, Q. Guo, M. Zhu, Subcritical approach to sharp Hardy-Littlewood-Sobolev type inequalities on the upper half space, Adv. Math. 312(2017), 1-45, 2017; Corrigendum to "Subcritical approach to sharp Hardy-Littlewood-Sobolev type inequalities on the upper half space" [Adv. Math. 312: 1-45, 2017], Adv. Math. 317(2017), 640-644.

\bibitem{Dou-Zhu2019}J. Dou, M. Zhu, Nonlinear integral equations on bounded domains, J. Funct. Anal., 277 (2019),111-134.

\bibitem{DZ2015}J. Dou, M. Zhu, Reversed Hardy-Littewood-Sobolev inequality, Int. Math. Res. Not., 19 (2015),9696-9726.

\bibitem{DZ2015a} J. Dou, M. Zhu, Sharp Hardy-Littlewood-Sobolev inequality on the upper half space, Int. Math. Res. Not., 3 (2015), 651-687.

\bibitem{Dragomir-Tomassini2006} S. Dragomir, G. Tomassini, Differential geometry and analysis on CR manifolds, Birkh\"{a}user, Boston, 2006.



\bibitem{Folland-Stein1974}G.B. Folland, E.M. Stein, \textsl{Estimates for the $\bar{\partial}_b$ complex and analysis on the Heisenberg group}, Comm. Pure Appl. Math.,  27(1974), 429-522.

\bibitem{Folland1975} G. B. Folland, {Subelliptic estimates and function spaces on nilpotent Lie groups}, Arkiv f\"{o}r Matematik, 13 (1975), 161-207.

\bibitem{Folland1973}G.B. Folland, A fundamental solution for a subelliptic operator, \textit{Bull. Amer. Math. Soc.}, 79(1973), 373-376.

\bibitem{Frank-Lieb2012} R. L. Frank, E.H. Lieb, {Sharp constants in several inequalities on the Heisenberg group}, Ann. of Math., 176(2012), 349-381.



\bibitem{Garofalo-Lanconelli1990}N. Garofalo, E. Lanconelli, {Frequency functions on the Heisenberg group, the uncertainty principle and unique continuation}, Ann. Inst. Fourier (Grenoble), 40 (1990), 313-356.


\bibitem{G1976} L. Gross, Logarithmic Sobolev inequalities, Amer. J. Math., 97 (1976), 1061-1083.
\bibitem{Guo2019}Q. Guo, Blowup analysis for integral equations on bounded domains, J. Diff. Equ., 266 (2019), 8258-8280.

\bibitem{H2013} X. Han, Existence of maximizers for Hardy-Littlewood-Sobolev inequalities on the Heisenberg group, Indiana Univ. Math. J., 62(3)(2013), 737-751.

\bibitem{HLZ2012} X. Han, G. Lu, J. Zhu, Hardy-Lilttlewood-Sobolev and Stein-Weiss inequalities and integral systems on the Heisenberg group, Nonl. Anal., 75(2012), 4296-4314.

\bibitem{Han-JDE2020} Y. Han, An integral type Brezis-Nirenberg problem on the Heisenberg group, J. Diff. Equ., 269(2020), 4544-4565. 
\bibitem{Han-arxiv2019} Y. Han, Integral equations on compact CR manifold, Discrete Contin. Dyn. Syst-A, 41(5)(2021), 2187-2204.

\bibitem{Han-Niu2005} Y. Han, P. Niu, Hardy-Sobolev type inequalities on the H-type group, Manuscripta Math. 118(2005), 235-252.

\bibitem{Han-Wang-Zhu2017}Y. Han, X. Wang, M. Zhu, Characterization by symmetry of solutions of a nonlinear subelliptic equation on the Heisenberg group, J. Math. Study, 50(1)(2017), 17-27.

\bibitem{Han-Zhang2020} Y. Han, S. Zhang, Sharp Sobolev inequalities on the complex sphere, Math. Ineq. Appl., 23 (2020),  149-159.

\bibitem{Han-Zhu2016} Y. Han, M. Zhu, Hardy-Littlewood-Sobolev inequalities on compact Riemannian manifolds and applications, J. Diff. Equ., 260 (2016), 1-25.

\bibitem{HWY2009}  F. Hang, X. Wang, X. Yan, An integral equation in conformal geometry, Ann. Inst. H.  Poincar\'{e} Analyse Non Lin\'{e}aire 26 (2009),  1-21.

\bibitem{HWY2007} F. Hang, X. Wang, X. Yan, Sharp integral inequalities for Harmonic functions, Comm. Pure Appl. Math., 61(1)(2007), 54-95.

\bibitem{HL1928} G. H. Hardy,  J. E. Littlewood, Some properties of fractional integrals (1), Math. Zeitschr. 27 (1928), 565-606.

\bibitem{HL1930} G. H. Hardy,  J. E. Littlewood, On certain inequalities connected with the calculus of variantions, J. London Math. Soc., 5(1930), 34-39



\bibitem{JL1987} D. Jerison, J. M. Lee, The Yamabe problem on CR manifolds, J. Diff. Geom. 25 (1987), 167-197.

\bibitem{JL1988} D. Jerison, J. M. Lee, Extremals for the Sobolev inequality on the Heisenberg group and the CR Yamabe problem, J. Amer. Math. Soc., 1 (1988), 1-13.


\bibitem{LP1987} J. M. Lee, T. H. Parker, The Yamabe problem, Bull. Amer. Math. Soc., 17 (1987), 37-91.

\bibitem{Li2004} Y. Y. Li, Remark on some conformally invariant integral equations: the method of moving spheres, J. Eur. Math. Soc.,  6 (2004), 153-180.



\bibitem{Lieb1983} E. H. Lieb, Sharp constants in the Hardy-Littlewood-Sobolev and related inequalities, Ann. of Math., 118 (1983), 349-374.

\bibitem{Lieb-Loss2001} E. H. Lieb, M. Loss, Analysis, Volume 14 of Graduate Studies in Mathematics, American Mathemaical Society, Providentce, RI, Second edition, 2001.

\bibitem{Moser1971} J. Moser, A sharp form of an inequality by N. Trudinger, Indiana Univ. Math. J. 20(1970/71), 1077-1092.

\bibitem{Ngo-Nguyen2107} Q. A. Ng\^{o}, V. Nguyen, Sharp reversed Hardy-Littlewood-Sobolev inequality on $\mathbb{R}^{n}$, Israel J. Math., 220( 2017), 189-223.

\bibitem{Ngo-Nguyen2017b} Q. A. Ng\^{o}, V. Nguyen, Sharp reversed Hardy-Littlewood-Sobolev inequality: The case of half space $\mathbb{R}^n_+$, Int. Math. Res. Not., 20 (2017), 6187-6230.


\bibitem{Niu-Zhang-Wang01} P. Niu, H. Zhang, Y. Wang,  Hardy type and Rellich type inequalities on the Heisenberg group, Proc. Amer. Math. Soc., 129(2001), 3623-3630.

\bibitem{Obata1971} M. Obata, The conjuctures on conformal transformations of Riemannian manifolds, J. Diff. Geom., 6(1971), 247-258.


\bibitem{So1963} S. L. Sobolev, On a theorem of functional analysis, Mat. Sb. (N.S.) 4 (1938), 471-479. A. M. S. Transl. Ser. 2, 34 (1963), 39-68.

\bibitem{Trudinger1967} N.S. Trudinger, On imbeddings into Orlicz spaces and some applications, J. Math. Mech., 17(1967), 473-483.

\bibitem{zhang-Han2}S. Zhang and Y. Han, Rearrangement free method for Hardy-Littlewood-Sobolev inequalities on $\mathbb{S}^n$, accepted by \textit{Analysis in Theory and Applications}.
\end{thebibliography}
\end{document}